\documentclass[12pt]{amsart}

\usepackage[utf8]{inputenc}
\usepackage[english]{babel}
\usepackage{amsmath,amsfonts,amssymb,amsthm,amscd}
\usepackage[a4paper,nomarginpar]{geometry}
\usepackage{epsfig,graphicx,color}
\usepackage[all]{xy}

\allowdisplaybreaks

\newtheorem{theorem}{Theorem}
\newtheorem{lemma}[theorem]{Lemma}
\newtheorem{proposition}[theorem]{Proposition}
\newtheorem{corollary}[theorem]{Corollary}

\theoremstyle{definition}

\newtheorem{examples}[theorem]{Examples}

\newtheorem{remark}[theorem]{Remark}

\newcommand{\N}{\mathbb{N}}  % set of natural numbers
\newcommand{\Z}{\mathbb{Z}}  % set of integer numbers
\newcommand{\R}{\mathbb{R}}  % set of real numbers
\newcommand{\C}{\mathbb{C}}  % set of complex numbers
\newcommand{\D}{\mathbb{D}}  % unit disc
\newcommand{\K}{\mathbb{K}}  % field

\newcommand{\eps}{\varepsilon} % abbreviation for epsilon
\newcommand{\ov} {\overline}   % abbreviation for overline
\newcommand{\cR}{\mathcal{R}}

\newcommand{\cV}{\mathcal{V}}

\newcommand{\spa}{\operatorname{span}}

\newcommand{\co}{\operatorname{co}}
\newcommand{\Orb}{\operatorname{Orb}}
\newcommand{\card}{\operatorname{card}}

\newcommand{\udens}{\operatorname{\overline{dens}}}
\newcommand{\rec}{\operatorname{Rec}}
\newcommand{\norm}[1]{{\left\|#1\right\|}}	% Norm notation

\begin{document}

\title[On shadowing and chain recurrence in linear dynamics]{On shadowing and chain recurrence\\ in linear dynamics}

\date{} %\today
\dedicatory{}
\maketitle

\begin{center}
{\sc Nilson C. Bernardes Jr.}

\medskip
Institut Universitari de Matem\`atica Pura i Aplicada\\
Universitat Polit\`ecnica de Val\`encia\\
Edifici 8E, Acces F, 4a Planta, 46022 Val\`encia, Spain\\
and\\
Departamento de Matem\'atica Aplicada\\
Universidade Federal do Rio de Janeiro\\
Caixa Postal 68530, 21945-970 Rio de Janeiro, Brazil

\smallskip
{\it e-mail}: ncbernardesjr@gmail.com

\bigskip
{\sc Alfredo Peris}

\medskip
Institut Universitari de Matem\`atica Pura i Aplicada\\
Universitat Polit\`ecnica de Val\`encia\\
Edifici 8E, Acces F, 4a Planta, 46022 Val\`encia, Spain

\smallskip
{\it e-mail}: aperis@mat.upv.es
\end{center}

\medskip
\begin{abstract}
In the present work we study the concepts of shadowing and chain recurrence in the setting of linear dynamics.
We prove that shadowing and finite shadowing always coincide for operators on Banach spaces,
but we exhibit operators on the Fr\'echet space $H(\C)$ of entire functions that have the finite shadowing property
but do not have the shadowing property.
We establish a characterization of mixing for continuous maps with the finite shadowing property in the setting of uniform spaces,
which implies that chain recurrence and mixing coincide for operators with the finite shadowing property on any topological vector space.
We establish a characterization of dense distributional chaos for operators with the finite shadowing property on Fr\'echet spaces.
As a consequence, we prove that if a Devaney chaotic (resp.\ a chain recurrent) operator on a Fr\'echet space
(resp.\ on a Banach space) has the finite shadowing property, then it is densely distributionally chaotic.
We obtain complete characterizations of chain recurrence for weighted shifts on Fr\'echet sequence spaces.
We prove that generalized hyperbolicity implies periodic shadowing for operators on Banach spaces.
Moreover, the concepts of shadowing and periodic shadowing coincide for unilateral weighted backward shifts,
but these notions do not coincide in general, even for bilateral weighted shifts.
\end{abstract}

%%%%%%%%%%%%%%%%%%%%%%%%%%%%%%%%%%%%%%%%%%%%%%%%%%%%%%%%%%%%%%%

\section{Introduction}

Consider a metric space $X$ with metric $d$ and a map $f : X \to X$.
Given $\delta > 0$, recall that a {\em $\delta$-pseudotrajectory} of $f$ is a finite or infinite sequence $(x_j)_{i < j < k}$ in $X$,
where $-\infty \leq i < k \leq \infty$ and $k - i \geq 3$, such that
\[
d(f(x_j),x_{j+1}) \leq \delta \ \ \text{ for all } i < j < k-1.
\]
A finite $\delta$-pseudotrajectory of the form $(x_j)_{j=0}^k$ is also called a {\em $\delta$-chain for $f$ from $x_0$ to $x_k$}
and the number $k$ is its {\em length}.
Recall that $f$ has the {\em positive shadowing property} if for every $\eps > 0$, there exists $\delta > 0$ such that every
$\delta$-pseudotrajectory $(x_j)_{j \in \N_0}$ of $f$ is {\em $\eps$-shadowed} by a real trajectory of $f$, that is,
there exists $x \in X$ such that
\[
d(x_j,f^j(x)) < \eps \ \ \text{ for all } j \in \N_0.
\]
If $f$ is bijective, then the {\em shadowing property} is defined by replacing the set $\N_0$ by the set $\Z$ in the above definition.
Recall also that $f$ is {\em chain recurrent} (resp.\ {\em chain transitive}) if for every $x \in X$ (resp.\ $x,y \in X$) and
every $\delta > 0$, there is a $\delta$-chain for $f$ from $x$ to itself (resp.\ from $x$ to $y$).
Moreover, $f$ is {\em chain mixing} if for every $x,y \in X$ and every $\delta > 0$, there exists $k_0 \in \N$ such that
for every $k \geq k_0$, there is a $\delta$-chain for $f$ from $x$ to $y$ with length $k$.

The notions of pseudotrajectory, shadowing and chain recurrence originated in the semi\-nal works of Conley \cite{CCon72},
Sina$\breve{\text{\i}}$ \cite{JSin72} and Bowen \cite{RBow75} in the early 1970's.
These concepts play a fundamental role in the qualitative theory of dynamical systems and differential equations.
We refer the reader to the books \cite{NAokKHir94,RDev89,AKatBHas95,KPal00,SPil99,MShu87} for nice expositions on these
important concepts and their applications.

In the last few years some interesting results on shadowing and chain recurrence were obtained in the setting of linear dynamics
\cite{AlvBerMes21,AntManVar22,BerCirDarMesPuj18,NBerAMes21,CirGolPuj21}.
For instance, it has long been known that every invertible hyperbolic operator on a Banach space
has the shadowing property~\cite{AMor81,JOmb94}
(an operator on a Banach space is said to be {\em hyperbolic} if its spectrum does not intersect the unit disc)
and that the converse holds in the finite-dimensional setting~\cite{AMor81,JOmb94}
and for invertible normal operators on Hilbert spaces~\cite{MMaz00}.
However, it remained open for a while whether this converse is always true or not.
This problem was finally settled in~\cite{BerCirDarMesPuj18}, where the first examples of non-hyperbolic operators
with the shadowing property were exhibited.
In the present work we will continue this line of investigation by analyzing some problems on shadowing and chain recurrence
for operators.
Although our main goal is to investigate the dynamics of linear operators on Fr\'echet spaces and, in particular, on Banach spaces,
some of our results will be established in much greater generality.
Below we present the topics covered in the paper and its organization.

In Section~\ref{SVFS} we will consider the finite shadowing property. This variation of the notion of shadowing is defined
as the positive shadowing property but considering only finite pseudotrajectories (of arbitrary length) instead of infinite ones.
From the computational point of view, it seems to be even more relevant than shadowing, since computer-generated trajectories
are actually {\em finite} pseudotrajectories.
It is well known that shadowing and finite shadowing coincide in the setting of compact metric spaces \cite[Lemma~1.1.1]{SPil99},
but this equivalence already fails on a certain locally compact subspace of $\R$ \cite[Example~2.3.4]{DarGonSob21}.
We will investigate the validity of the equivalence between shadowing and finite shadowing in the setting of linear dynamics.
Our main result asserts that these concepts always coincide for operators on Banach spaces (Theorem~\ref{EquivSh}).
Nevertheless, we will exhibit operators on the Fr\'echet space $H(\C)$ of entire functions that have the finite shadowing property
but do not have the shadowing property (Theorem~\ref{Counterexample}).

In Section~\ref{Chaos} we will investigate some chaotic behaviors of operators with the finite sha\-dowing property.
We will establish a characterization of mixing for continuous maps with the finite shadowing property in the setting of uniform spaces
(Theorem~\ref{Equi1}), which will imply a very general theorem in linear dynamics (Theorem~\ref{Equi2}).
We will also establish a characterization of dense distributional chaos for operators with the finite shadowing property
on Fr\'echet spaces (Theorem~\ref{DistChaos}).
As applications, we will show that if a Devaney chaotic (resp.\ a chain recurrent) continuous linear operator on a Fr\'echet space
(resp.\ on a Banach space) has the finite shadowing property, then it is densely distributionally chaotic
(Theorems~\ref{DCDDC} and \ref{CRDDC}).
In particular, the Devaney chaotic operators constructed by Menet~\cite{QMen17} do not have the finite shadowing property,
since they are not distributionally chaotic.

In Section~\ref{WeightedShifts} we will consider weighted shifts on Fr\'echet sequence spaces.
Due to the importance of weighted shifts in the area of operator theory and its applications, the dynamics of these operators
has been extensively investigated by many researchers (see the books \cite{FBayEMat09,KGroAPer11} and the papers
\cite{FBayIRuz15,BerBonMulPer13,BerBonPer20,BerCirDarMesPuj18,NBerAMes21,GCosMSam04,KGro00,HSal95}, for instance).
Our goal in this section is to establish complete characterizations of chain recurrence for weighted shifts
on Fr\'echet sequence spaces (Theorems~\ref{BBSCRThm}, \ref{BWBSCRThm}, \ref{UBSCRThm} and \ref{UWBSCRThm}),
which extend previous results from \cite{AlvBerMes21} in the case of $c_0$ and $\ell^p$ spaces.
Moreover, we will illustrate these characterizations by presenting concrete examples on some classical sequence spaces.

In Section~\ref{PeriodicShadowing} we will investigate the so-called periodic shadowing property \cite{PKos05,OsiPilTik10}
for continuous linear operators on Banach spaces.
Our main result asserts that generalized hyperbolicity implies periodic shadowing (Theorem~\ref{PS} and Corollary~\ref{PSCor}).
Next we will prove that positive shadowing and positive periodic shadowing coincide for unilateral weighted backward shifts
on the classical Banach sequence spaces $\ell_p(\N)$ ($1 \leq p < \infty$) and $c_0(\N)$ (Theorem~\ref{Characterization1}).
However, the notions of shadowing and periodic shadowing do not coincide in general, even for bilateral weighted shifts.
In fact, we will obtain a class of operators with the periodic shadowing property (Theorem~\ref{PeriodicThm})
that includes bilateral weighted shifts without the shadowing property (Corollary~\ref{BWBSNot}).

In the Appendix at the end of the paper we will establish several basic facts related to the notion of chain recurrence and
the shadowing property for continuous linear operators on topological vector spaces. Our goal is to lay the foundations in
great generality, complementing and extending previous basic results on this subject
(see \cite{AlvBerMes21,AntManVar22,BerCirDarMesPuj18,NBerAMes21,SPil99}, for instance).
Since these results are of a more elementary character, we decided to postpone them to the Appendix.
However, some of these results will be used in previous sections, but properly referenced.

We will close the paper by proposing some open problems.

Throughout $\K$ denotes either the field $\R$ of real numbers or the field $\C$ of complex numbers.
Moreover, $\N$ denotes the set of all positive integers and $\N_0\!:= \N \cup \{0\}$.
Whenever we consider a Fr\'echet space $X$, we will tacitly assume that we have already chosen an increasing sequence
$(\|\cdot\|_k)_{k \in \N}$ of seminorms that induces its topology and that it is endowed with the compatible complete
invariant metric given by
\begin{equation}\label{Metric}
d(x,y)\!:= \sum_{k=1}^\infty \frac{1}{2^k} \min\{1,\|x - y\|_k\} \ \ \ \ (x,y \in X).
\end{equation}
We observe that the notions of shadowing and chain recurrence depend only on the underlying uniform structure of the space,
and so they do not depend on the specific compatible invariant metric we choose.

%%%%%%%%%%%%%%%%%%%%%%%%%%%%%%%%%%%%%%%%%%%%%%%%%%%%%%%%%%%%%%%

\section{Shadowing versus finite shadowing for operators}\label{SVFS}

In this section we will investigate whether or not shadowing and finite shadowing coincide for operators on Fr\'echet spaces.

Given a metric space $X$, recall that a map $f : X \to X$ has the {\em finite shadowing property}
if for every $\eps > 0$, there exists $\delta > 0$ such that for each $\delta$-chain $(x_j)_{j=0}^k$ of $f$, there exists $x \in X$ with
\[
d(x_j,f^j(x)) < \eps \ \ \text{ for all } j \in \{0,\ldots,k\}.
\]
In this case, if $f$ is bijective, $(x_j)_{j=-i}^k$ is a finite $\delta$-pseudotrajectory of $f$ ($i \geq 0$, $k \geq 1$) and we define
$y_t\!:= x_{t-i}$ for $t \in \{0,\ldots,i+k\}$, then $(y_t)_{t=0}^{i+k}$ is a $\delta$-chain for $f$, and so there exists
$y \in X$ with $d(y_t,f^t(y)) < \eps$ for all $t \in \{0,\ldots,i+k\}$. Hence, $x\!:= f^i(y) \in X$ satisfies
\[
d(x_j,f^j(x)) < \eps \ \ \text{ for all } j \in \{-i,\ldots,k\}.
\]
This explains why we use the terminology ``finite shadowing'' instead of ``positive finite shadowing'' for the above notion.

It turns out that shadowing and finite shadowing always coincide for operators on Banach spaces.

\begin{theorem}\label{EquivSh}
For any invertible continuous linear operator $T$ on any Banach space $X$, the following assertions are equivalent:
\begin{itemize}
\item [(i)] $T$ has the shadowing property;
\item [(ii)] $T$ has the positive shadowing property;
\item [(iii)] $T$ has the finite shadowing property.
\end{itemize}
In the non-invertible case, (ii) and (iii) are equivalent.
\end{theorem}

\begin{proof}
(i) $\Rightarrow$ (ii) $\Rightarrow$ (iii): Obvious.

\noindent
(iii) $\Rightarrow$ (i) :
Given $\eps > 0$, let $\delta > 0$ be associated to $\eps/4$ according to the fact that $T$ has the finite shadowing property.
Let $(x_j)_{j \in \Z}$ be a $\delta$-pseudotrajectory of $T$. The first step consists in replacing the $\delta$-pseudotrajectory
$(x_j)_{j \in \Z}$ by a $\delta$-pseudotrajectory $(y_j)_{j \in \Z}$ which is close to $(x_j)_{j \in \Z}$, in the sense that
\begin{equation}\label{FS1}
\|y_j - x_j\| < \frac{\eps}{4} \ \ \text{ for all } j \in \Z,
\end{equation}
and has the following additional property:
\begin{equation}\label{FS2}
\lim_{j \to \pm\infty} \|Ty_j - y_{j+1}\| = 0.
\end{equation}
For this purpose, choose $m \in \N$ with $\eps/m < \delta$ and define
\[
n_k\!:= \frac{k(k-1)m}{2} \ \ \text{ for all } k \in \N.
\]
For each $k \in \N$, there exists $u_k \in X$ such that
\begin{equation}\label{FS3}
\|x_j - T^ju_k\| < \frac{\eps}{4} \ \ \text{ for all } j \in \{-n_{k+1},\ldots,n_{k+1}\}.
\end{equation}
For each $k \in \N$ and each $j \in \{0,\ldots,km-1\}$, define
\[
y_{n_k + j}\!:= \frac{km-j}{km}\, T^{n_k + j}u_k + \frac{j}{km}\, T^{n_k + j}u_{k+1}
\]
and
\[
y_{-n_k - j}\!:= \frac{km-j}{km}\, T^{-n_k - j}u_k + \frac{j}{km}\, T^{-n_k - j}u_{k+1}.
\]
It follows from (\ref{FS3}) that (\ref{FS1}) holds. Moreover, for each $k \in \N$ and each $j \in \{0,\ldots,km-1\}$,
\begin{align*}
\|Ty_{n_k+j} - y_{n_k+j+1}\| &= \Big\|\frac{1}{km}\, T^{n_k+j+1}u_k - \frac{1}{km}\, T^{n_k+j+1}u_{k+1}\Big\|\\
  &\leq \frac{1}{km}\, \Big\|T^{n_k+j+1}u_k - x_{n_k+j+1}\Big\| + \frac{1}{km}\, \Big\|x_{n_k+j+1} - T^{n_k+j+1}u_{k+1}\Big\|\\
  &< \frac{\eps}{2km} < \delta
\end{align*}
and, similarly,
\[
\|Ty_{-n_k-j} - y_{-n_k-j+1}\| < \frac{\eps}{2km} < \delta.
\]
This shows that $(y_j)_{j \in \Z}$ is also a $\delta$-pseudotrajectory of $T$ and that (\ref{FS2}) holds.

The second step consists in constructing inductively an increasing sequence $(m_k)_{k \in \N}$ of positive integers,
a sequence $(v_k)_{k \in \N}$ of vectors in $X$ and a sequence $((y^{(k)}_j)_{j \in \Z})_{k \in \N}$ of pseudotrajectories of $T$
satisfying the following conditions for each $k \in \N$:
\begin{itemize}
\item [(a)] $(y^{(k)}_j)_{j \in \Z}$ is a $\displaystyle\frac{\delta}{2^{k-1}}$-pseudotrajectory of $T$;

\smallskip
\item [(b)] $\displaystyle \lim_{j \to \pm\infty} \|Ty^{(k)}_j - y^{(k)}_{j+1}\| = 0$;

\smallskip
\item [(c)] $\displaystyle\|T y^{(k)}_j - y^{(k)}_{j+1}\| < \frac{\delta}{2^{k+1}}$ whenever $|j| \geq m_k$;

\smallskip
\item [(d)] $\displaystyle\|y^{(k)}_j - T^jv_k\| < \frac{\eps}{2^{k+1}}$ whenever $|j| \leq m_k + p$;

\smallskip
\item [(e)] $y^{(k)}_0 = v_{k-1}$ and $\displaystyle\|y^{(k)}_j - y^{(k-1)}_j\| < \frac{\eps}{2^k}$ for all $j \in \Z$ (provided $k \geq 2$).
\end{itemize}
The number $p$ is a fixed positive integer greater than $\eps/\delta$. We begin by defining
\[
y^{(1)}_j\!:= y_j \ \ \text{ for all } j \in \Z.
\]
By (\ref{FS2}), we can choose an $m_1 \in \N$ such that
\[
\|Ty^{(1)}_j - y^{(1)}_{j+1}\| < \frac{\delta}{2^2} \ \ \text{ whenever } |j| \geq m_1.
\]
By finite shadowing, there exists $v_1 \in X$ such that
\[
\|y^{(1)}_j - T^jv_1\| < \frac{\eps}{2^2} \ \ \text{ whenever } |j| \leq m_1 + p.
\]
Hence, (a), (b), (c) and (d) hold with $k\!:= 1$. Suppose that $m_k$, $v_k$ and $(y^{(k)}_j)_{j \in \Z}$ have already been chosen
for $k \in \{1,\ldots,t\}$ so that all the desired properties hold. Define
\[
y^{(t+1)}_j\!:=
\left\{\begin{array}{cll}
T^jv_t & \text{if} & |j| \leq m_t\\
\frac{m_t + p - |j|}{p}\, T^j v_t + \frac{|j| - m_t}{p} \, y^{(t)}_j & \text{if} & m_t < |j| < m_t + p\\
y^{(t)}_j & \text{if} & |j| \geq m_t + p
\end{array}\right. \!\!.
\]
Some elementary computations show that (a) and (e) hold with $k\!:= t+1$. Since
\[
\lim_{j \to \pm\infty} \|Ty^{(t+1)}_j - y^{(t+1)}_{j+1}\| = \lim_{j \to \pm\infty} \|Ty^{(t)}_j - y^{(t)}_{j+1}\| = 0,
\]
we have that (b) holds with $k\!:= t+1$. Moreover, we can choose an $m_{t+1} > m_t$ such that (c) hold with $k\!:= t+1$.
Finally, by finite shadowing, there exists a vector $v_{t+1} \in X$ such that (d) holds with $k\!:= t+1$.
This completes our induction process.

Let us now complete the proof. Since
\[
\|v_k - v_{k+1}\| = \|y^{(k+1)}_0 - T^0v_{k+1}\| < \frac{\eps}{2^{k+2}} \ \ \text{ for all } k \in \N,
\]
we have that $(v_k)_{k \in \N}$ is a Cauchy sequence in $X$. By completeness, there exists
\[
v\!:= \lim_{k \to \infty} v_k \in X.
\]
Moreover,
\[
\|T^jv_k - y_j\| \leq \|T^jv_k - y^{(k)}_j\| + \sum_{t=2}^{k} \|y^{(t)}_j - y^{(t-1)}_j\|
  < \frac{\eps}{2^{k+1}} + \sum_{t=2}^{k} \frac{\eps}{2^t} < \frac{\eps}{2}\,,
\]
whenever $k \in \N$ and $|j| \leq m_k + p$. By fixing $j \in \Z$ and letting $k \to \infty$, we obtain
\begin{equation}\label{FS4}
\|T^jv - y_j\| \leq \frac{\eps}{2} \ \ \text{ for all } j \in \Z.
\end{equation}
By (\ref{FS1}) and (\ref{FS4}), the $\delta$-pseudotrajectory $(x_j)_{j \in \Z}$ is $\eps$-shadowed by the trajectory of $v$,
proving that $T$ has the shadowing property.

In the non-invertible case, the arguments are analogous.
\end{proof}

One important fact in the previous proof was that the $\eps$-$\delta$ association can be selected to be linear if we have a Banach space, essential for applying the induction process in the second step of the proof to get (d) with $\frac{\eps}{2^{k+1}}$ from finite shadowing when we have a $\frac{\delta}{2^{k-1}}$-pseudotrajectory. This is something that we cannot do in general with an $F$-norm for a Fréchet space. Actually, the equivalence between shadowing and finite shadowing may fail for ope\-rators on Fr\'echet spaces.

\begin{theorem}\label{Counterexample}
Let $H(\C)$ be the Fr\'echet space of all entire functions endowed with the compact-open topology.
For each $\lambda \in \C$ with $|\lambda| \not\in \{0,1\}$, the multiplication operator
\[
M_\lambda : f \in H(\C) \mapsto \lambda f \in H(\C)
\]
has the finite shadowing property but does not have the shadowing property.
\end{theorem}

\begin{proof}
For each $k \in \N$, let
\[
D_k\!:= \{z \in \C : |z| < k\} \ \ \ \text{ and } \ \ \ \|f\|_k\!:= \sup_{z \in D_k} |f(z)| \ \text{ for } f \in H(\C).
\]
The sequence $(\|\cdot\|_k)_{k \in \N}$ of seminorms induces the compact-open topology on $H(\C)$.
Consider $H(\C)$ endowed with its canonical metric given by (\ref{Metric}).
For each $k \in \N$, let $A(D_k)$ be the ``disk algebra on the disk $D_k$'', that is,
\[
A(D_k)\!:= \{g : \ov{D_k} \to \C : g \text{ is continuous on } \ov{D_k} \text{ and analytic on } D_k\}
\]
endowed with the norm $\|\cdot\|_k$, which is a Banach space (actually, a Banach algebra).
In view of Proposition~\ref{InverseShad} in the Appendix, it is enough to consider the case $|\lambda| > 1$,
since $M_{\lambda^{-1}} = (M_\lambda)^{-1}$.
Thus, fix $\lambda \in \C$ with $|\lambda| > 1$ and let $T\!:= M_\lambda$.

Let us prove that $T$ has the finite shadowing property. For this purpose, fix $\eps > 0$ and choose $\ell \in \N$ such that
\[
d(f,0) < \eps \ \ \text{ whenever } f \in H(\C) \text{ and } \|f\|_\ell < \frac{\eps}{2}\,\cdot
\]
Let
\[
S : g \in A(D_\ell) \mapsto \lambda g \in A(D_\ell).
\]
Since $S$ is a proper dilation on the Banach space $A(D_\ell)$ (i.e., $\|S^{-1}\| < 1$), $S$ is a hyperbolic operator.
Hence, $S$ has the shadowing property. Let $\eta > 0$ be such that every $\eta$-pseudotrajectory of $S$ is ($\eps/2$)-shadowed
by a real trajectory of $S$. Choose $\delta > 0$ such that
\[
\|f\|_\ell \leq \eta \ \ \text{ whenever } f \in H(\C) \text{ and } d(f,0) \leq \delta.
\]
If $(f_j)_{j=0}^k$ is a $\delta$-chain for $T$, then $(f_j|_{\ov{D_\ell}})_{j=0}^k$ is an $\eta$-chain for $S$,
and so there exists $g \in A(D_\ell)$ such that
\[
\|f_j|_{\ov{D_\ell}} - S^j g\|_\ell < \frac{\eps}{2} \ \ \text{ for all } j \in \{0,\ldots,k\}.
\]
By the density of the polynomials in $A(D_\ell)$, there is a polynomial $f$ so close to $g$ in $A(D_\ell)$ that we have
\[
\|f_j - T^j f\|_\ell = \|f_j|_{\ov{D_\ell}} - S^j f\|_\ell < \frac{\eps}{2} \ \ \text{ for all } j \in \{0,\ldots,k\}.
\]
Thus, $d(f_j,T^j f) < \eps$ for all $j \in \{0,\ldots,k\}$, as it was to be shown.

Now, suppose that $T$ has the shadowing property. Let $\delta > 0$ be associated to $\eps\!:= 1/2$ according to this property.
Choose $\ell \in \N$ such that
\[
d(f,0) < \delta \ \ \text{ whenever } f \in H(\C) \text{ and } \|f\|_\ell < \frac{\delta}{2}\,\cdot
\]
Choose a function $g \in A(D_\ell)$ that cannot be extended to an entire function.
We shall construct inductively a sequence $(f_j)_{j \in \N_0}$ of polynomials such that:
\begin{itemize}
\item [(A)] $\displaystyle \|\lambda^j g - f_j\|_\ell < \frac{\delta}{2 |\lambda|}$ for all $j \in \N_0$;
\item [(B)] $\displaystyle \|\lambda f_{j-1} - f_j\|_\ell < \frac{\delta}{2}$ for all $j \in \N$.
\end{itemize}
We begin by choosing a polynomial $f_0$ with $\|g - f_0\|_\ell < \frac{\delta}{2|\lambda|}$.
Assume $k \in \N_0$ and $f_0,\ldots,f_k$ already chosen with the desired properties. Since
\[
\|\lambda^{k+1} g - \lambda f_k\|_\ell = |\lambda| \|\lambda^k g - f_k\|_\ell < \frac{\delta}{2}\,,
\]
there is a polynomial $p_k$ so close to $\lambda^{k+1} g - \lambda f_k$ in $A(D_\ell)$ that we have
\[
\|p_k\|_\ell < \frac{\delta}{2} \ \ \text{ and } \ \
\|\lambda^{k+1} g - \lambda f_k - p_k\|_\ell < \frac{\delta}{2 |\lambda|}\,\cdot
\]
Hence, it is enough to define $f_{k+1}\!:= \lambda f_k + p_k$.
By (B), $(f_j)_{j \in \N_0}$ is a $\delta$-pseudotrajectory of~$T$.
Therefore, there exists $f \in H(\C)$ such that
\[
d(f_j,T^j f) < \frac{1}{2} \ \ \text{ for all } j \in \N_0.
\]
This implies that
\[
\|f_j - \lambda^j f\|_1 < 1 \ \ \text{ for all } j \in \N_0.
\]
By (A), we obtain
\[
\|\lambda^j g - \lambda^j f\|_1 < 1 + \frac{\delta}{2 |\lambda|} \ \ \text{ for all } j \in \N_0.
\]
Thus, $g = f$ on $D_1$. By the principle of analytic continuation, $g = f$ on $\ov{D_\ell}$.
This contradicts our choice of $g$ as an element of $A(D_\ell)$ that cannot be extended to an entire function.
Our conclusion is that $T$ does not have the shadowing property.
\end{proof}

\begin{remark}\label{Remarks}
(a) The above proof actually shows that $M_\lambda$ does not have the positive sha\-dowing property whenever $|\lambda| > 1$.
Thus, the notions of finite shadowing and positive shadowing do not coincide in general for invertible operators on Fr\'echet spaces.\\
(b) For $0 < |\lambda| < 1$, the operator $M_\lambda$ has the positive shadowing property.
Indeed, this follows easily from the fact that if $\ell \in \N$ and if a sequence $(f_j)_{j \in \N_0}$ in $H(\C)$ satisfies
\[
\|M_\lambda f_j - f_{j+1}\|_\ell \leq \delta \ \ \text{ for all } j \in \N_0,
\]
then
\[
\|f_j - (M_\lambda)^j f_0\|_\ell \leq \frac{\delta}{1 - |\lambda|} \ \ \text{ for all } j \in \N_0.
\]
Thus, the notions of shadowing and positive shadowing do not coincide in general for invertible operators on Fr\'echet spaces.
Since $(M_\lambda)^{-1} = M_{\lambda^{-1}}$ does not have the positive shadowing property, this also shows that we cannot
replace shadowing by positive sha\-dowing in Proposition~\ref{InverseShad}.
\end{remark}

\begin{remark}
We observe that completeness is essential for the validity of Theorem~\ref{EquivSh}.
For instance, let $X$ be the vector space of all sequences $(x_n)_{n \in \N}$ of scalars with finite support endowed
with any $\ell_p$-norm ($1 \leq p \leq \infty$) and let $T \in GL(X)$ be twice the identity operator on $X$.
Since twice the identity operator on $c_0(\N)$ or $\ell_p(\N)$ ($1 \leq p < \infty$) has the shadowing property
(because it is hyperbolic), it follows that $T$ has the finite shadowing property.
However, $T$ does not have the positive shadowing property.
In fact, given any $\delta > 0$, consider the sequence $(x^{(j)})_{j \in \N_0}$ in $X$ given by
\[
x^{(0)}\!:= 0 \ \ \text{ and } \ \ x^{(j)}\!:= (2^{j-1}\delta,2^{j-2}\delta,\ldots,2\delta,\delta,0,0,\ldots) \text{ for } j \geq 1.
\]
Then $(x^{(j)})_{j \in \N_0}$ is a $\delta$-pseudotrajectory of $T$, but it cannot be $1$-shadowed by a trajectory of $T$
since each element of $X$ has finite support.
\end{remark}

%%%%%%%%%%%%%%%%%%%%%%%%%%%%%%%%%%%%%%%%%%%%%%%%%%%%%%%%%%%%%%%

\section{Chaotic behaviors in the presence of the finite shadowing property}\label{Chaos}

Our goal in this section is to investigate some types of chaotic behavior for operators with the finite shadowing property.

We begin by recalling some notions of chaotic behavior. Let $X$ be a topological space and $f : X \to X$ a map.
Given sets $A,B \subset X$, the {\em return set of $f$ from $A$ to $B$} is defined~by
\[
N_f(A,B)\!:= \{n \in \N_0 : f^n(A) \cap B \neq \emptyset\}.
\]
Recall that $f$ is {\em topologically transitive} (resp.\ {\em topologically ergodic}, {\em topologically mixing})
if for any pair $A,B$ of nonempty open subsets of $X$, the return set $N_f(A,B)$ is nonempty (resp.\ syndetic, cofinite), where
a set $I\!:= \{n_1 < n_2 < \cdots\} \subset \N_0$ is {\em syndetic} when it has bounded gaps, that is, $\sup_{k} (n_{k+1} - n_k) < \infty$.
Moreover, $f$ is {\em topologically weakly mixing} if $f \times f$ is topologically transitive, that is,
$N_f(A_1,B_1) \cap N_f(A_2,B_2) \neq \emptyset$ for any $4$-tuple $A_1,A_2,B_1,B_2$ of nonempty open subsets of $X$.
In the sequel we will omit the word ``topologically'' from these notions.
If $X$ is a second countable Baire space without isolated points and $f$ is continuous,
then {\em Birkhoff's transitivity theorem} asserts that $f$ is transitive if and only if it admits a dense orbit,
that is, there exists a point $x \in X$ whose orbit $\Orb(x,f)\!:= \{f^n(x) : n \in \N_0\}$ is dense in $X$.

In the setting of linear dynamics, the existence of a dense orbit is known under the name of {\em hypercyclicity}.
Hence, a continuous linear operator $T$ on a topological vector space $X$ is {\em hypercyclic} if it admits a dense orbit.
Recall also that $T$ is {\em Devaney chaotic} if it is transitive and has a dense set of periodic points.
Hypercyclic and Devaney chaotic operators have been extensively studied during the last 30 years.
We refer the reader to the books \cite{FBayEMat09,KGroAPer11} for an overview of the area of linear dynamics up to 2010.

Our first theorem in this section will give us a characterization of mixing for continuous maps with finite shadowing
in the setting of uniform spaces.
In order to state and prove the theorem, let us first recall the notion of chain transitivity and the finite shadowing property
in this more general setting. We refer the reader to \cite[Chapter~II]{NBou1989} for the basics on uniform spaces.

Consider a uniform space $X$ with uniformity $\mathcal{U}$ and a map $f : X \to X$.
Given  $V\in\mathcal{U}$, a {\em $V$-chain for $f$} is a finite sequence $(x_j)_{j=0}^k$ in $X$ satisfying
\[
(f(x_j),x_{j+1}) \in V \ \ \text{ for all } n \in \{0,\ldots,k-1\}.
\]
In this case, we also say that $(x_j)_{j=0}^k$ is a {\em $V$-chain for $f$ from $x_0$ to $x_k$}.
The map $f$ has the {\em finite shadowing property} if for every $V\in\mathcal{U}$, there exists  $U\in \mathcal{U}$
such that for each $U$-chain $(x_j)_{j=0}^k$ for $f$, there exists $x \in X$ with
\[
(x_j,f^j(x)) \in V \ \ \text{ for all } j \in \{0,\ldots,k\}.
\]
A point $x \in X$ is a {\em chain recurrent point} of $f$ if for every $V\in \mathcal{U}$, there is a $V$-chain for $f$ from $x$ to itself.
The set $CR(f)$ of all chain recurrent points of $f$ is called the {\em chain recurrent set} of $f$ and $f$ is said to be {\em chain recurrent}
if $CR(f) = X$. Moreover, $f$ is said to be {\em chain transitive} if for every $x,y \in X$ and every $V\in\mathcal{U}$,
there is a $V$-chain for $f$ from $x$ to $y$.
We recall that for each point $x \in X$, the sets of the form
\[
V(x)\!:= \{y \in X : (x,y) \in V\},
\]
as $V$ runs through the uniformity $\mathcal{U}$, constitute a fundamental system of neighborhoods of $x$ in $X$.
We also recall that $A \subset X$ is a {\em $V$-small set} if $A \times A \subset V$.

\begin{theorem}\label{Equi1}
Consider a uniform space $X$ with uniformity $\mathcal{U}$ and a continuous map $f : X \to X$.
If $f$ has the finite shadowing property, then $f$ is mixing if and only if the following conditions hold:
\begin{itemize}
\item [(I)] $f$ is chain transitive;
\item [(II)] For each $V\in \mathcal{U}$, there is a $V$-small set $A \subset X$ with $N_f(A,A)$ cofinite.
\end{itemize}
\end{theorem}

\begin{proof}
Since the necessity of the conditions is clear, let us prove their sufficiency.
Let $A$ and $B$ be nonempty open sets in $X$. Choose points $x \in A$ and $y \in B$, and let $V \in \mathcal{U}$ be such that
\[
V(x) \subset A \ \ \ \text{ and } \ \ \ V(y) \subset B.
\]
Let $U \in \mathcal{U}$ be associated to $V$ according to the definition of the finite shadowing property,
and let $W \in \mathcal{U}$ satisfy
\[
W \circ W\!:= \{(a,c) : (a,b) \in W \text{ and } (b,c) \in W \text{ for some } b \in X\} \subset U.
\]
By condition (II), there is a $W$-small set $Z \subset X$ such that $N_f(Z,Z)$ is cofinite.
Choose a point $z \in Z$ and let $m \in \N$ be such that $n \in N_f(Z,Z)$ for all $n \geq m$.
By condition (I), there exist $W$-chains $(x_j)_{j=0}^k$ and $(y_j)_{j=0}^\ell$ for $f$ from $x$ to $z$ and from $z$ to $y$, respectively.
Given $n \geq m$, there exists $z' \in Z$ such that $f^n(z') \in Z$. Hence,
\[
(u_j)_{j=0}^{k+n+\ell}\!:= (x_0,x_1,\ldots,x_{k-1},z',f(z'),\ldots,f^{n-1}(z'),y_0,y_1,\ldots,y_\ell)
\]
is a $U$-chain for $f$ from $x$ to $y$, and so there exists $u \in X$ such that
\[
(u_j,f^j(u)) \in V \ \ \text{ for all } j \in \{0,\ldots,k+n+\ell\}.
\]
In particular, $u \in A$ and $f^{k+n+\ell}(u) \in B$. This proves that
\[
f^t(A) \cap B \neq \emptyset \ \ \text{ for all } t \geq t_0,
\]
where $t_0\!:= k+m+\ell$.
\end{proof}

\begin{remark}
Condition (II) in Theorem~\ref{Equi1} cannot be omitted in general.
For instance, consider $X\!:= \{0,1\}$ endowed with its discrete uniformity and let $f : X \to X$ be given by $f(0)\!:= 1$ and $f(1)\!:= 0$.
Then $f$ is chain transitive and has the shadowing property, but it is not mixing.
\end{remark}

As an application of the above theorem, we obtain the following result on linear dynamics.

\begin{theorem}\label{Equi2}
Suppose that a continuous linear operator $T$ on a topological vector space $X$ has the finite shadowing property.
Then the following assertions are equivalent:
\begin{itemize}
\item [(i)] $T$ is chain recurrent;
\item [(ii)] $T$ is transitive;
\item [(iii)] $T$ is ergodic;
\item [(iv)] $T$ is weakly mixing;
\item [(v)] $T$ is mixing.
\end{itemize}
Moreover, $T$ is Devaney chaotic if and only if $T$ has a dense set of periodic points.
\end{theorem}

\begin{proof}
Recall that a basis for the uniformity of $X$ is given by the sets
\[
\widetilde{V}\!:= \{(x,y) \in X \times X : x - y \in V\},
\]
as $V$ runs through the set of all neighborhoods of $0$ in $X$.

In order to prove the equivalences from (i) to (v), it is enough to show that (i) implies~(v).
This follows from Theorem~\ref{Equi1}, because condition~(II) is automatically true in the present case
(note that $N_T(A,A) = \N_0$ whenever $0 \in A$) and Proposition~\ref{CR-CT-CM} in the Appendix gives condition~(I).

Now, suppose that $T$ has a dense set of periodic points. Since the set $CR(T)$ of all chain recurrent points of $T$ is closed in $X$
(Proposition~\ref{InvClosedSubsp}), we have that $T$ is chain recurrent.
Hence, by (i)~$\Rightarrow$~(ii), $T$ is transitive, and so it is Devaney chaotic.
\end{proof}

The above theorem extends \cite[Theorem~3.3]{AntManVar22} and \cite[Corollary~3.9]{AntManVar22} from normed spaces
and separable Banach spaces, respectively, to arbitrary topological vector spaces, but we observe that the arguments in
\cite{AntManVar22} also work in the more general context.

Our next goal is to establish a characterization of dense distributional chaos for operators with the finite shadowing property
in the setting of Fr\'echet spaces. For this purpose, let us recall the notion of distributional chaos in metric spaces and some
related concepts in the context of linear dynamics.

Given a metric space $X$, recall that $f : X \to X$ is said to be {\em distributionally chaotic} if there exist
an uncountable set $\Gamma \subset X$ and an $\eps > 0$ such that each pair $(x,y)$ of distinct elements of $\Gamma$
is an {\em $\eps$-distributionally chaotic pair} for $f$, in the sense that
\[
\udens\{n \in \N : d(f^n(x),f^n(y)) \geq \eps\} = 1
\]
and
\[
\udens\{n \in \N : d(f^n(x),f^n(y)) < \delta\} = 1 \ \text{ for all } \delta > 0,
\]
where $\udens(I)$ stands for the {\em upper density} of the subset $I$ of $\N$, that is,
\[
\udens(I)\!:= \limsup_{n \to \infty} \frac{\card(I \cap [1,n])}{n}\,\cdot
\]

If $X$ is a Fr\'echet space whose topology is induced by an increasing sequence $(\|\cdot\|_k)_{k \in \N}$ of seminorms and
$T : X \to X$ is a continuous linear operator, recall that $x \in X$ is called a {\em distributionally irregular vector} for $T$
if there exist $m \in \N$ and $I,J \subset \N$ with $\udens(I) = \udens(J) = 1$ such that
\[
\lim_{n \in I} T^n x = 0 \ \ \ \text{ and } \ \ \ \lim_{n \in J} \|T^n x\|_m = \infty.
\]
It was proved in \cite{BerBonMulPer13} that:
\[
T \textit{ is distributionally chaotic} \ \Leftrightarrow \ T \textit{ admits a distributionally irregular vector}.
\]
We shall follow the terminology of \cite{GriMatMen21} and say that $T$ is {\em densely distributionally chaotic} if it admits a dense
set of distributionally irregular vectors. It follows from results in \cite{BerBonMulPer13} that this is equivalent to the existence of a
{\em residual set} of distributionally irregular vectors.

Before stating our theorem, let us introduce the following notations: For each continuous linear operator $T$ on a Fr\'echet space $X$,
we define the sets
\begin{align*}
I_0(T) &\!:= \{x \in X : \text{for each } \delta > 0,\, \text{there is a $\delta$-chain for $T$ from $x$ to $0$}\},\\
O_0(T) &\!:= \{x \in X : \text{for each } \delta > 0,\, \text{there is a $\delta$-chain for $T$ from $0$ to $x$}\}.
\end{align*}

\begin{lemma}\label{I0O0}
The sets $I_0(T)$ and $O_0(T)$ are $T$-invariant closed subspaces of $X$.
\end{lemma}

As a consequence, if both $I_0(T)$ and $O_0(T)$ are dense in~$X$, then $T$ is chain recurrent.
The converse is also true, because the notions of chain recurrence and chain transitivity coincide for linear operators
(Proposition~\ref{CR-CT-CM}). We leave the proof of the above lemma to the reader.

\begin{theorem}\label{DistChaos}
Suppose that a continuous linear operator $T$ on a Fr\'echet space $X$ has the finite shadowing property.
Then $T$ is densely distributionally chaotic if and only if the following conditions hold:
\begin{itemize}
\item [(I)] $I_0(T)$ is dense in $X$;
\item [(II)] There exists $\gamma > 0$ such that for every $\delta > 0$, there is a $\delta$-chain for $T$ from $0$ to a vector
$x \in X$ satisfying
\[
\udens\{j \in \N_0 : d(T^jx,0) \geq \gamma\} = 1.
\]
\end{itemize}
If $X$ is a Banach space, then we can replace condition (II) by the following weaker condition:
\begin{itemize}
\item [(II')] There exists $\gamma > 0$ such that for every $\delta > 0$, there is a $\delta$-chain for $T$ from $0$ to a vector
$x \in X$ satisfying
\[
\udens\{j \in \N_0 : d(T^jx,0) \geq \gamma\} \geq \gamma.
\]
\end{itemize}
\end{theorem}

\begin{proof}
Suppose that $T$ is densely distributionally chaotic.
Since $T$ admits a dense set of vectors whose trajectories have subsequences converging to $0$, condition~(I) holds.
Moreover, if $y$ is a distributionally irregular vector for $T$, then there exists $\gamma > 0$ such that
\[
\udens\{j \in \N_0 : d(T^jy,0) \geq \gamma\} = 1.
\]
Given $\delta > 0$, there exists $n \in \N$ such that $d(T^ny,0) < \delta$, and so $(0,T^ny)$ is a $\delta$-chain for $T$
from $0$ to the vector $x\!:= T^ny$ which satisfies $\udens\{j \in \N_0 : d(T^jx,0) \geq \gamma\} = 1$,
proving that condition~(II) also holds.

\smallskip
Conversely, suppose that conditions (I) and (II) hold. For each $k \in \N$, let
\[
A_k\!:= \{x \in X : \exists n \in \N \text{ with}\, \card\{1 \leq j \leq n : d(T^j x,0) < k^{-1}\} \geq n(1 - k^{-1})\}.
\]
It is clear that each $A_k$ is open in $X$. Fix $k \in \N$, $x \in X$ and $\eps > 0$. Let $\eta\!:= \min\{\eps,k^{-1}\}$ and let $\delta > 0$
be associated to $\eta$ according to the fact that $T$ has the finite shadowing property.
By condition~(I), there is a $\delta$-chain $(x_j)_{j=0}^t$ for $T$ from $x$ to $0$.
Choose $n \in \N$ such that $n - t > n(1 - k^{-1})$ and define $x_j\!:= 0$ for all $j \in \{t+1,\ldots,n\}$.
Since $(x_j)_{j=0}^n$ is a $\delta$-chain for $T$, there exists $y \in X$ such that
\[
d(x_j,T^j y) < \eta \ \ \text{ for all } j \in \{0,\ldots,n\}.
\]
Hence, $y \in A_k$ and $d(y,x) < \eps$. This proves that $A_k$ is dense in $X$.
It follows that the set $R_1$ of all $x \in X$ for which there exists $I \subset \N$ with
\[
\udens(I) = 1 \ \ \ \text{ and } \ \ \ \lim_{n \in I} T^n x = 0
\]
is residual in $X$.

Now, let $\eps\!:= \gamma/2$ and, for each $k \in \N$, let $\delta_k > 0$ be associated to $\eps/k$ according to the
finite shadowing property. By condition~(II), for each $k \in \N$, there is a $\delta_k$-chain $(x_{k,j})_{j=0}^{t_k}$
for $T$ from $0$ to a vector $x_k \in X$ satisfying
\[
\udens(I_k) = 1, \ \text{ where } I_k\!:= \{j \in \N_0 : d(T^jx_k,0) \geq \gamma\}.
\]
For each $k \in \N$, consider a natural number $N_k > t_k$ and define
\[
x_{k,j}\!:= T^{j-t_k}x_k \ \ \text{ for all } j \in \{t_k+1,\ldots,N_k\}.
\]
Since $(x_{k,j})_{j=0}^{N_k}$ is a $\delta_k$-chain for $T$, there exists $y_k \in X$ such that
\[
d(x_{k,j},T^j y_k) < \eps/k \ \ \text{ for all } j \in \{0,\ldots,N_k\}.
\]
If $j \in \{t_k+1,\ldots,N_k\}$ and $j - t_k \in I_k$, then
\[
d(T^jy_k,0) \geq d(x_{k,j},0) - d(x_{k,j},T^jy_k) > \gamma - \eps/k \geq \eps.
\]
Since $\udens(I_k) = 1$, we can choose $N_k$ as large as we want so that
\begin{equation}\label{EqDC2}
\card\{1 \leq j \leq N_k : d(T^jy_k,0) > \eps\} > N_k(1 - k^{-1}).
\end{equation}
Hence, we can choose $N_1 < N_2 < N_3 < \cdots$ so that (\ref{EqDC2}) holds for all $k \in \N$. Thus,
\[
\lim_{k \to \infty} y_k = 0 \ \ \text{ and } \ \ \lim_{k \to \infty} \frac{1}{N_k} \card\{1 \leq j \leq N_k : d(T^j y_k,0) > \eps\} = 1.
\]
By \cite[Proposition~8]{BerBonMulPer13}, the set $R_2$ of all $x \in X$ for which there are $m \in \N$ and $J \subset \N$ with
\[
\udens(J) = 1 \ \ \ \text{ and } \ \ \ \lim_{n \in J} \|T^n x\|_m = \infty
\]
is residual in $X$.
Therefore, the set $R_1 \cap R_2$ is also residual in $X$.
Since each element of this set is a distributionally irregular vector for $T$, we conclude that $T$ is densely distributionally chaotic.

\smallskip
Let us now assume that $X$ is a Banach space and that conditions (I) and (II') hold.
Let $R_1$ and $R_2$ be as above.
Since condition~(I) was not changed, $R_1$ is residual in $X$.
By following the arguments used in the previous paragraph with $\udens(I_k) \geq \gamma$ instead of $\udens(I_k) = 1$,
we see that we can choose $N_k$ as large as we want so that
\begin{equation}\label{EqDC3}
\card\{1 \leq j \leq N_k : d(T^jy_k,0) > \eps\} > \eps N_k.
\end{equation}
Hence, we can choose $N_1 < N_2 < N_3 < \cdots$ so that (\ref{EqDC3}) holds for all $k \in \N$.
By \cite[Proposition~8]{BerBonMulPer13} in the case of Banach spaces, we conclude that $R_2$ is residual in $X$.
\end{proof}

\begin{remark}
If $T$ is a proper contraction (respectively, a proper dilation) on a Banach space $X$, then $T$ has the shadowing property
and condition (I) (respectively, condition (II)) holds, but $T$ is not distributionally chaotic.
This shows that each one of the conditions in Theorem~\ref{DistChaos} is essential for its validity.
\end{remark}

As applications of the previous theorem, we obtain the following results.

\begin{theorem}\label{DCDDC}
If a Devaney chaotic continuous linear operator $T$ on a Fr\'echet space $X$ has the finite shadowing property,
then it is densely distributionally chaotic.
\end{theorem}

\begin{proof}
Actually, it is not difficult to check that Devaney chaos implies conditions (I) and (II) of the previous theorem.
\end{proof}

\begin{theorem}\label{CRDDC}
If a chain recurrent continuous linear operator $T$ on a Banach space $X$ has the finite shadowing property,
then it is densely distributionally chaotic.
\end{theorem}

\begin{proof}
Since $T$ is chain transitive (Proposition~\ref{CR-CT-CM}), condition (I) holds.
By Theorem~\ref{EquivSh}, $T$ has the positive shadowing property.
Let $\eta > 0$ be associated to $\eps\!:= 1$ according to this property.
Choose a vector $y \in X$ with $\|y\| \geq 2$ and let $(y_j)_{j=0}^k$ be an $\eta$-chain for $T$ from $y$ to itself.
Since
\[
(y_0,y_1,\ldots,y_k,y_1,\ldots,y_k,y_1,\ldots,y_k,\ldots)
\]
is an $\eta$-pseudotrajectory of $T$, there exists $x \in X$ such that
\[
\|y - T^{nk}x\| < 1 \ \ \text{ for all } n \in \N_0.
\]
Thus,
\[
\udens\{j \in \N_0 : \|T^jx\| \geq 1\} \geq 1/k.
\]
Since $T$ is chain transitive, we conclude that condition (II') holds with $\gamma\!:= 1/k$.
\end{proof}

%%%%%%%%%%%%%%%%%%%%%%%%%%%%%%%%%%%%%%%%%%%%%%%%%%%%%%%%%%%%%%%

\section{Chain recurrent weighted shifts on Fr\'echet sequence spaces}\label{WeightedShifts}

Our goal in this section is to characterize the notion of chain recurrence for weighted shifts on Fr\'echet sequence spaces.

\smallskip
Recall that a {\em Fr\'echet sequence space} is a Fr\'echet space $X$ which is a vector subspace of the product space $\K^\N$
such that the inclusion map $X \to \K^\N$ is continuous, that is, convergence in $X$ implies coordinatewise convergence.
Given a sequence $w\!:= (w_n)_{n \in \N}$ of nonzero scalars, it follows from the closed graph theorem that
the {\em unilateral weighted backward shift}
\[
B_w(x_1,x_2,x_3,\ldots)\!:= (w_2x_2,w_3x_3,w_4x_4,\ldots)
\]
is a continuous linear operator on $X$ as soon as it maps $X$ into itself.
The {\em canonical vectors} of $\K^\N$ are the vectors $e_n$, $n \in \N$, where the $n^\text{th}$ coordinate of $e_n$ is $1$
and the other coordinates of $e_n$ are $0$. The sequence $(e_n)_{n \in \N}$ is a {\em basis} of $X$ if each $e_n$ belongs to $X$ and
\[
x = \sum_{n=1}^\infty x_ne_n \ \ \text{ for all } x\!:= (x_n)_{n \in \N} \in X.
\]

We will also consider Fr\'echet sequence spaces consisting of bilateral sequences.
A {\em Fr\'echet sequence space over $\Z$} is a Fr\'echet space $X$ which is a vector subspace of the product space $\K^\Z$
such that the inclusion map $X \to \K^\Z$ is continuous. As before, if $w\!:= (w_n)_{n \in \Z}$ is a sequence of nonzero scalars,
then the {\em bilateral weighted backward shift}
\[
B_w((x_n)_{n \in \Z})\!:= (w_{n+1}x_{n+1})_{n \in \Z}
\]
is a continuous linear operator on $X$ as soon as it maps $X$ into itself.
By abuse of language, we also denote the {\em canonical vectors} of $\K^\Z$ by $e_n$, $n \in \Z$.
The sequence $(e_n)_{n \in \Z}$ is a {\em basis} of $X$ if each $e_n$ belongs to $X$ and
\[
x = \sum_{n=-\infty}^\infty x_ne_n \ \ \text{ for all } x\!:= (x_n)_{n \in \Z} \in X.
\]

In the results below, we are adopting the convention that $c/0 = \infty$ whenever $c \in (0,\infty)$.

\smallskip
We begin by characterizing chain recurrence for bilateral (unweighted) backward shifts.

\begin{theorem}\label{BBSCRThm}
Suppose that $X$ is a Fr\'echet sequence space over $\Z$ in which the sequence $(e_n)_{n \in \Z}$ of canonical vectors is a basis,
$(\|\cdot\|_k)_{k \in \N}$ is an increasing sequence of seminorms that induces the topology of $X$,
and the bilateral backward shift
\[
B : (x_n)_{n \in \Z} \in X \mapsto (x_{n+1})_{n \in \Z} \in X
\]
is a well-defined operator. Then $B$ is chain recurrent if and only if
\begin{equation}\label{BBSCR1}
\sum_{n=1}^\infty \frac{1}{\|e_{-n}\|_k} = \sum_{n=1}^\infty \frac{1}{\|e_n\|_k} = \infty \ \ \text{ for all } k \in \N.
\end{equation}
\end{theorem}

\begin{proof}
Consider $X$ endowed with its canonical metric given by (\ref{Metric}).

Suppose that (\ref{BBSCR1}) holds. In view of Lemma~\ref{I0O0}, the chain recurrence of $B$ follows from the two claims below.

\medskip\noindent
{\bf Claim 1.} $e_i \in O_0(B)$ for all $i \in \N$.

\medskip
Indeed, fix $i \in \N$ and $\delta > 0$. Choose $\ell \in \N$ such that
\begin{equation}\label{BBSCR2}
d(x,0) < \delta \ \ \text{ whenever } x \in X \text{ and } \|x\|_\ell < \delta/2.
\end{equation}
By hypothesis,
\begin{equation}\label{BBSCR3}
\sum_{n=1}^\infty \frac{1}{\|e_n\|_k} = \infty \ \ \text{ for all } k \in \N.
\end{equation}

Suppose that there exists $n_k \in \N$ with $\|e_{n_k}\|_k = 0$, for each $k \in \N$.
Choose $k \geq \ell$ such that $n_k > i$. By (\ref{BBSCR2}),
\[
0,e_{n_k},e_{n_k-1},\ldots,e_i
\]
is a $\delta$-chain for $B$ from $0$ to $e_i$, proving that $e_i \in O_0(B)$.

Now, suppose that there exists $k_0 \in \N$ such that $\|e_n\|_{k_0} \neq 0$ for all $n \in \N$.
Choose $k \in \N$ with $k \geq \max\{k_0,\ell\}$. By (\ref{BBSCR3}), there exists $m \in \N$ such that
\[
t\!:= \sum_{n=i+1}^{i+m} \frac{1}{\|e_n\|_k} > \frac{2}{\delta}\,\cdot
\]
Define
\[
x_1\!:= \frac{e_{i+m}}{t \|e_{i+m}\|_k} \ \ \ \text{ and } \ \ \
x_{j+1}\!:= Bx_j + \frac{e_{i+m-j}}{t \|e_{i+m-j}\|_k} \ \text{ for } 1 \leq j < m.
\]
Note that $x_m = e_{i+1}$. Hence, it follows from (\ref{BBSCR2}) that
\[
0,x_1,x_2,\ldots,x_m,e_i
\]
is a $\delta$-chain for $B$ from $0$ to $e_i$, proving that $e_i \in O_0(B)$.

\medskip\noindent
{\bf Claim 2.} $e_{-i} \in I_0(B)$ for all $i \in \N$.

\medskip
Indeed, let $i$, $\delta$ and $\ell$ be as in the proof of Claim~1. By hypothesis,
\begin{equation}\label{BBSCR4}
\sum_{n=1}^\infty \frac{1}{\|e_{-n}\|_k} = \infty \ \ \text{ for all } k \in \N.
\end{equation}

If there exists $n_k \in \N$ with $\|e_{-n_k}\|_k = 0$, for each $k \in \N$,
then we choose $k \geq \ell$ with $n_k > i$ and apply (\ref{BBSCR2}) to conclude that
\[
e_{-i},e_{-i-1},\ldots,e_{-n_k+1},0
\]
is a $\delta$-chain for $B$ from $e_{-i}$ to $0$, proving that $e_{-i} \in I_0(B)$.

If there exists $k_0 \in \N$ such that $\|e_{-n}\|_{k_0} \neq 0$ for all $n \in \N$,
then we choose $k \in \N$ with $k \geq \max\{k_0,\ell\}$. By (\ref{BBSCR4}), there exists $m \in \N$ such that
\[
t\!:= \sum_{n=-i-m}^{-i-1} \frac{1}{\|e_n\|_k} > \frac{2}{\delta}\,\cdot
\]
Define
\[
x_0\!:= e_{-i} \ \ \ \text{ and } \ \ \
x_j\!:= Bx_{j-1} - \frac{e_{-i-j}}{t \|e_{-i-j}\|_k} \ \text{ for } 1 \leq j \leq m.
\]
Note that $x_m = 0$. Hence, it follows from (\ref{BBSCR2}) that
\[
x_0,x_1,x_2,\ldots,x_m
\]
is a $\delta$-chain for $B$ from $e_{-i}$ to $0$, proving that $e_{-i} \in I_0(B)$.

\medskip
Conversely, suppose that $B$ is chain recurrent. Let us first prove that (\ref{BBSCR3}) holds. For this purpose, fix $k \in \N$.
We may assume that $\|e_n\|_k \neq 0$ for all $n \in \N$, for otherwise the desired equality would hold trivially.
By the Banach-Steinhaus theorem, there exists $\delta > 0$ such that
\begin{equation}\label{BBSCR5}
x\!:= (x_n)_{n \in \Z} \in X \text{ and } d(x,0) < \delta \ \ \Longrightarrow \ \ \|x_ne_n\|_k < 1 \text{ for all } n \in \Z.
\end{equation}
Fix $r > 0$. By hypothesis, there is a $\delta$-chain $(z_j)_{j=0}^m$ for $B$ from $r e_1$ to itself.
For each $j \in \{1,\ldots,m\}$, let $y_j\!:= z_j - Bz_{j-1}$ and write $y_j = (y_{j,n})_{n \in \Z}$. By (\ref{BBSCR5}),
\begin{equation}\label{BBSCR6}
|y_{j,m-j+1}| < \frac{1}{\|e_{m-j+1}\|_k} \ \ \text{ for all } j \in \{1,\ldots,m\}.
\end{equation}
Since
\begin{equation}\label{BBSCR7}
z_m = B^m z_0 + B^{m-1} y_1 + B^{m-2} y_2 + \cdots + B y_{m-1} + y_m,
\end{equation}
we obtain
\begin{align*}
r &= y_{1,m} + y_{2,m-1} + \cdots + y_{m-1,2} + y_{m,1}\\
  &\leq |y_{1,m}| + |y_{2,m-1}| + \cdots + |y_{m-1,2}| + |y_{m,1}|\\
  &< \frac{1}{\|e_m\|_k} + \frac{1}{\|e_{m-1}\|_k} + \cdots + \frac{1}{\|e_2\|_k} + \frac{1}{\|e_1\|_k}\,,
\end{align*}
where in the last inequality we used (\ref{BBSCR6}). Since $r > 0$ is arbitrary, we are done.

\smallskip
Let us now establish (\ref{BBSCR4}).
Assume $k \in \N$, $\|e_{-n}\|_k \neq 0$ for all $n \in \N$, $\delta > 0$ such that (\ref{BBSCR5}) holds, and $r > 0$.
Let $(z_j)_{j=0}^m$ be a $\delta$-chain for $B$ from $r e_0$ to itself.
For each $j \in \{1,\ldots,m\}$, let $y_j\!:= z_j - Bz_{j-1}$ and write $y_j = (y_{j,n})_{n \in \Z}$. By (\ref{BBSCR5}),
\begin{equation}\label{BBSCR8}
|y_{j,-j}| < \frac{1}{\|e_{-j}\|_k} \ \ \text{ for all } j \in \{1,\ldots,m\}.
\end{equation}
Since (\ref{BBSCR7}) holds, we obtain $0 = r + y_{1,-1} + y_{2,-2} + \cdots + y_{m,-m}$. Thus, by (\ref{BBSCR8}),
\[
r \leq |y_{1,-1}| + |y_{2,-2}| + \cdots + |y_{m,-m}|
  < \frac{1}{\|e_{-1}\|_k} + \frac{1}{\|e_{-2}\|_k} + \cdots + \frac{1}{\|e_{-m}\|_k}\,\cdot
\]
Since $r > 0$ is arbitrary, the proof is complete.
\end{proof}

The previous theorem can be generalized to bilateral weighted backward shifts as follows.

\begin{theorem}\label{BWBSCRThm}
Suppose that $X$ is a Fr\'echet sequence space over $\Z$ in which the sequence $(e_n)_{n \in \Z}$ of canonical vectors is a basis,
$(\|\cdot\|_k)_{k \in \N}$ is an increasing sequence of seminorms that induces the topology of $X$,
$w\!:= (w_n)_{n \in \Z}$ is a sequence of nonzero scalars, and the bilateral weighted backward shift
\[
B_w : (x_n)_{n \in \Z} \in X \mapsto (w_{n+1}x_{n+1})_{n \in \Z} \in X
\]
is a well-defined operator. Then $B_w$ is chain recurrent if and only if
\begin{equation}\label{BWBSCR1}
\sum_{n=1}^\infty \frac{1}{|w_{-n+1} \cdots w_0| \|e_{-n}\|_k}
  = \sum_{n=1}^\infty \frac{|w_1 \cdots w_n|}{\|e_n\|_k} = \infty \ \ \text{ for all } k \in \N.
\end{equation}
\end{theorem}

\smallskip
The above theorem follows from the previous one by means of a suitable conjugacy.
The method can be found in \cite[Section~4.1]{KGroAPer11}, but we will recall it here briefly for the sake of completeness.
Consider the weights
\[
v_0\!:= 1, \ \ \ \ v_{-n}\!:= w_{-n+1} \cdots w_0 \ \ \text{ and } \ \ v_n\!:= \frac{1}{w_1 \cdots w_n} \ \text{ for } n \geq 1,
\]
the vector space
\[
X_v\!:= \{(x_n)_{n \in \Z} \in \K^\Z : (v_nx_n)_{n \in \Z} \in X\},
\]
and the vector space isomorphism
\[
\phi_v : (x_n)_{n \in \Z} \in X_v \mapsto (v_nx_n)_{n \in \Z} \in X.
\]
Use $\phi_v$ to transfer the topology of $X$ to $X_v$:
$U \subset X_v$ is declared to be open in $X_v$ if and only if $\phi_v(U)$ is open in $X$.
Then $X_v$ is a Fr\'echet space whose topology is induced by the sequence of seminorms given by
\[
\|(x_n)_{n \in \Z}\|'_k\!:= \|\phi_v((x_n)_{n \in \Z})\|_k \ \ \text{ for } (x_n)_{n \in \Z} \in X_v.
\]
Moreover, $(e_n)_{n \in \Z}$ is a basis of $X_v$ and $B_w \circ \phi_v = \phi_v \circ B$, that is, $\phi_v$ establishes a
conjugacy between $B_w$ and $B$. Hence, $B_w$ is chain recurrent if and only if so is $B$.
Thus, Theorem~\ref{BWBSCRThm} follows from Theorem~\ref{BBSCRThm} applied to $X_v$ endowed with the seminorms $\|\cdot\|'_k$.

\smallskip
Let us now consider the case of unilateral (unweighted) backward shifts.

\begin{theorem}\label{UBSCRThm}
Suppose that $X$ is a Fr\'echet sequence space in which the sequence $(e_n)_{n \in \N}$ of canonical vectors is a basis,
$(\|\cdot\|_k)_{k \in \N}$ is an increasing sequence of seminorms that induces the topology of $X$,
and the unilateral backward shift
\[
B : (x_1,x_2,x_3,\ldots) \in X \mapsto (x_2,x_3,x_4,\ldots) \in X
\]
is a well-defined operator. Then $B$ is chain recurrent if and only if
\begin{equation}\label{UBSCR1}
\sum_{n=1}^\infty \frac{1}{\|e_n\|_k} = \infty \ \ \text{ for all } k \in \N.
\end{equation}
\end{theorem}

\begin{proof}
If (\ref{UBSCR1}) holds, then the proof of Claim~1 in Theorem~\ref{BBSCRThm} shows that $e_i \in O_0(B)$ for all $i \in \N$.
Since $B$ is a unilateral backward shift, this implies that $B$ is chain recurrent.
The converse is proved as in the penultimate paragraph of the proof of Theorem~\ref{BBSCRThm}.
\end{proof}

\smallskip
The above theorem can be generalized to unilateral weighted backward shifts as follows.

\begin{theorem}\label{UWBSCRThm}
Suppose that $X$ is a Fr\'echet sequence space in which the sequence $(e_n)_{n \in \N}$ of canonical vectors is a basis,
$(\|\cdot\|_k)_{k \in \N}$ is an increasing sequence of seminorms that induces the topology of $X$,
$w\!:= (w_n)_{n \in \N}$ is a sequence of nonzero scalars, and the unilateral weighted backward shift
\[
B_w : (x_1,x_2,x_3,\ldots) \in X \mapsto (w_2x_2,w_3x_3,w_4x_4,\ldots) \in X
\]
is a well-defined operator. Then $B_w$ is chain recurrent if and only if
\[
\sum_{n=1}^\infty \frac{|w_1 \cdots w_n|}{\|e_n\|_k} = \infty \ \ \text{ for all } k \in \N.
\]
\end{theorem}

\smallskip
As in the case of bilateral shifts, Theorem~\ref{UWBSCRThm} can be easily deduced from Theorem~\ref{UBSCRThm}
by means of a suitable conjugacy.

\smallskip

The characterization of transitivity for weighted shifts on Fréchet sequence spaces was obtained in \cite{KGro00} (see also 4.1 in \cite{KGroAPer11}).
Actually, under the above notation, the unilateral weighted shift $B_w$ is transitive on $X$ if and only if there exists an increasing sequence $(n_k)_k$ in $\N$ tending to infinity such that
\[
 \frac{1}{|w_1 \cdots w_{n_k}|}\, e_{n_k} \to 0 \ \text{ in } X.
\]
In the case of a bilateral weighted shift, transitivity is equivalent to the existence of an increasing sequence $(n_k)_k$ in $\N$ tending to infinity such that
\[
 \frac{1}{|w_{j+1} \cdots w_{j+n_k}|}\, e_{j+n_k} \to 0 \ \mbox{ and } \ |w_{j+1-n_k} \cdots w_{j}|\, e_{j-n_k} \to 0,
\]
in $X$, for any $j \in \Z$.

Since transitivity obviously implies chain recurrence, the main interesting examples of chain recurrent weighted shifts are those that are not transitive. We will provide some natural examples, and to do so we need to recall the concept of Köthe sequence spaces (see \cite{kothe1969topological, meise_vogt1997introduction}), which can be viewed as an intersection of a decreasing sequence of weighted $\ell^p$-spaces, for a fixed $p$, or weighted $c_0$-spaces, when the matrix below consists of non-zero weights:

An infinite matrix $A\!:= (a_{j,k})_{j,k \in \N}$ of non-negative weights is a {\em K\"{o}the matrix} if, for each $j \in \N$
there exists a $k \in \N$ with $a_{j,k} > 0$, and $0 \le a_{j,k} \le a_{j,k+1}$ for all $j,k \in \N$.
Given $1 \le p < \infty$, we consider the Fr\'{e}chet sequence spaces
\[
\lambda_p(A)\!:= \big\{x \in \K^\N : \norm{x}_k\!:= \left(\sum_{j=1}^{\infty} |x_j|^p a_{j,k}^p \right)^{1/p} < \infty, \ \forall \ k \in \N \big\},
\]
and for $p = 0$,
\[
\lambda_0(A)\!:= \big\{x \in \K^\N : \lim_{j \to \infty} x_j a_{j,k} = 0, \norm{x}_k\!:= \sup_{j \in \N} |x_j| a_{j,k}, \ \forall \ k \in \N \big\},
\]
which are the associated K\"{o}the sequence spaces.

K\"{o}the spaces are certainly a natural class of Fr\'{e}chet sequence spaces.
Obviously, if the entries $a_{j,k}=1$ for all $j,k \in \N$, then we have $\lambda_p(A) = \ell^p$ and $\lambda_0(A) = c_0$.

For instance, the derivative operator $D$ of many function spaces $X$ can be represented as a weighted backward shift if the
Taylor representation around $0$ of functions $f \in X$ allows an isomorphism of $X$ with a K\"{o}the space.

In order to have that a unilateral  weighted backward shift $B_w$ is well-defined (equivalently, continuous) on a Köthe
sequence space, we need to consider some conditions that relate the weight sequence $w$ with the Köthe matrix $A$.
It is well known (see, e.g., \cite{martinez-gimenez_peris2002chaos}) that $B_w$ is continuous if and only if

\begin{equation}
 \forall \ n \in \N, \ \exists \ m > n \ : \ \sup_{i \in \N} w_{i+1} \frac{a_{i,n}}{a_{i+1,m}} < \infty. \label{cont_weight}
\end{equation}

\begin{examples}
(A) We consider three different Hilbert spaces of holomorphic functions on the unit disc.
Namely, the {\em Bergman space} $A^2$ of functions $f \in \mathcal{H}(\D)$ such that
\[
\norm{f}^2\!:= \frac{1}{\pi} \int_{\D} |f(z)|^2 d\lambda(z) < \infty,
\]
the {\em Dirichlet space} $\mathcal{D}$ of functions $f \in \mathcal{H}(\D)$ such that
\[
\norm{f}^2\!:= |f(0)|^2 + \frac{1}{\pi} \int_{\D} |f'(z)|^2 d\lambda(z) < \infty,
\]
where in both cases $\lambda$ denotes the two-dimensional Lebesgue measure,
and the {\em Hardy space} $H^2$ of functions $f \in \mathcal{H}(\D)$ such that
\[
\norm{f}^2\!:= \lim_{r \to 1^-}\frac{1}{2\pi} \int_0^{2\pi} |f(re^{it})|^2 dt < \infty.
\]
We have that $\mathcal{D} \subset H^2 \subset A^2$.
Moreover, via the identification with a sequence space by $f(z) = \sum_{n \geq 0} a_nz^n\mapsto (a_n)_n$, we know that
\[
\mathcal{D} = \ell^2(v) \mbox{ for } v\!:= (1,1,2,3,\dots), \ H^2 = \ell^2, \mbox{ and }
A^2 = \ell^2(v) \mbox{ for } v\!:= (1,\frac{1}{2},\frac{1}{3},\dots ),
\]
where
\[
\ell^2(v)\!:=\big\{a = (a_n)_n \in \C^{\N_0} : \norm{a}^2\!:=\sum_{n = 0}^\infty |a_n|^2v_n < \infty\big\}.
\]
A natural operator on these spaces is the (unweighted) backward shift that corresponds to
$(Bf)(z)\!:= (f(z) - f(0))/z$, $z \neq 0$, $(Bf)(0)\!:= f'(0)$.
The behavior concerning transitivity of $B$ is different on these spaces, since $B$ is transitive on the Bergman space $A^2$
by the above characterization, but it is not transitive on $\mathcal{D}$ or $H^2$ since $\norm{B} = 1$ in both spaces.
On the other hand, we have that
\[
\sum_{n = 0}^\infty \frac{1}{\norm{e_n}} = \sum_{n = 0}^\infty 1 = \infty \mbox{ on } H^2, \mbox{ and }
\sum_{n = 0}^\infty \frac{1}{\norm{e_n}} = 1 + \sum_{n = 1}^\infty \frac{1}{n} = \infty \mbox{ on } \mathcal{D},
\]
that is, $B$ is chain recurrent on these spaces.

\noindent
(B) Now we will consider non-normable Köthe spaces. For $A\!:= (k^j)_{j,k \in \N}$, we have that
$\lambda_p(A) = \lambda_2(A) = \mathcal{H}(\C)$.
If $A\!:= (j^k)_{j,k \in \N}$, we have that $\lambda_p(A) = \lambda_2(A)=:\! s$, the space of rapidly decreasing sequences,
and for the matrix $A\!:= (e^{-j/k})_{j,k \in \N}$, we have that $\lambda_p(A) = \lambda_2(A) = \mathcal{H}(\D)$,
the space of the holomorphic functions on the unit disc.
We obviously have that $\mathcal{H}(\C) \subset s \subset \mathcal{H}(\D)$.
The derivative operator $D$ corresponds to the weighted shift $B_w$ given by $w\!:= (1,2,3,\dots )$, and $D$ is transitive
(thus, chain recurrent) on the three spaces. If, as in (A), we consider the unweighted shift $B$, then
\[
\sum_{n=1}^\infty \frac{1}{\|e_n\|_k} < \infty \ \mbox{ for } \ k\geq 2,
\]
in $\mathcal{H}(\C)$ or $s$, so $B$ is not chain recurrent on these spaces. For the space $\mathcal{H}(\D)$ we have
\[
\sum_{n=1}^\infty \frac{1}{\|e_n\|_k} = \sum_{n=1}^\infty e^{n/k} = \infty \ \mbox{ for every } k \in \N,
\]
and $B$ is chain recurrent on $\mathcal{H}(\D)$. Actually, the transitivity condition is also fullfilled.
If we set $X\!:= \lambda_2(A)$ for $A\!:= ((\log(j+1))^k)_{j,k\in\N}$, then $s \subset X \subset \mathcal{H}(\D)$ and
\[
\sum_{n=1}^\infty \frac{1}{\|e_n\|_k} = \sum_{n=1}^\infty \frac{1}{(\log(n+1))^k} = \infty \ \mbox{ for every } k \in \N,
\]
so $B$ is chain recurrent on $X$ too, but the transitivity condition is not satisfied in this case.
\end{examples}

%%%%%%%%%%%%%%%%%%%%%%%%%%%%%%%%%%%%%%%%%%%%%%%%%%%%%%%%%%%%%%%

\section{Periodic shadowing for operators on Banach spaces}\label{PeriodicShadowing}

Our goal in this section is to investigate the notion of periodic shadowing for continuous linear operators on Banach spaces.

Given a metric space $X$, recall that $f : X \to X$ has the {\em positive periodic shadowing pro\-perty}
\cite{PKos05,OsiPilTik10} if for every $\eps > 0$, there exists $\delta > 0$ such that any periodic $\delta$-pseudotrajectory
$(x_n)_{n \in \N_0}$ of $f$ is $\eps$-shadowed by a periodic trajectory of $f$ (a sequence $(y_n)_{n \in \N_0}$ is
{\em periodic} if there exists $p \in \N$ such that $y_{n+p} = y_n$ for all $n \in \N_0$; such a $p$ is called a {\em period} for
sequence $(y_n)_{n \in \N_0}$). If $f$ is bijective, then the {\em periodic shadowing property} is defined by replacing the set
$\N_0$ by the set $\Z$ in the above definition.

Let us say that a continuous linear operator $T$ on a Banach space $X$ is {\em generalized hyperbolic} if there is
a direct sum decomposition
\[
X = M \oplus N,
\]
where $M$ and $N$ are closed subspaces of $X$ with the following properties ($r(T)$ denotes the spectral radius of $T$):

\smallskip \noindent
(a) $T(M) \subset M$ and $r(T|_M) < 1$;

\smallskip \noindent
(b) $T|_N : N \to T(N)$ is bijective, $T(N)$ is closed, $T(N) \supset N$ and $r((T|_N)^{-1}|_N) < 1$.

\smallskip
If $T$ is invertible, then condition (b) can be rewritten as follows:

\smallskip \noindent
(b') $T^{-1}(N) \subset N$ and $r(T^{-1}|_N) < 1$.

\smallskip
In the case of invertible operators, this class appeared in \cite{BerCirDarMesPuj18}, where it was proved that these operators have the
shadowing property, enabling the construction of the first examples of operators that have the shadowing property but are not hyperbolic.
The terminology ``generalized hyperbolic'' was introduced in \cite{CirGolPuj21}, where many additional dynamical properties of these
operators were investigated. The fact that every invertible genera\-lized hyperbolic operator is structurally stable was established
in \cite{NBerAMes20}. For not necessarily invertible operators, this class appeared in \cite{NBerAMesArxiv}.

It is known that generalized hyperbolic operators exhibit several types of shadowing properties
(see \cite{AlvBerMes21,BerCirDarMesPuj18,NBerAMesArxiv,JiaWanLi19}).
We shall now prove that they also have the periodic shadowing property.

\begin{theorem}\label{PS}
Every generalized hyperbolic operator $T$ on a Banach space $X$ has the posi\-tive periodic shadowing property.
\end{theorem}

\begin{proof}
Let $S$ denote the operator $(T|_N)^{-1}|_N$ on $N$.
For each $x \in X$, write $x = x^{(1)} + x^{(2)}$ with $x^{(1)} \in M$ and $x^{(2)} \in N$. Let $\alpha > 0$ be such that
\begin{equation}\label{PS1}
\|x^{(1)}\| \leq \alpha\, \|x\| \ \text{ and } \ \|x^{(2)}\| \leq \alpha\, \|x\| \ \text{ for all } x \in X.
\end{equation}
Since $r(T|_M) < 1$ and $r(S) < 1$, there exist $0 < t < 1$ and $\beta \geq 1$ such that
\begin{equation}\label{PS2}
\|T^ny\| \leq \beta\, t^n \|y\| \ \text{ and } \ \|S^nz\| \leq \beta\, t^n \|z\| \ \ \ \ (n \in \N_0, y \in M, z \in N).
\end{equation}
Given $\eps > 0$, put $\delta\!:= \frac{(1-t) \eps}{3 \alpha \beta}\cdot$ Let $(x_n)_{n \in \N_0}$ be a periodic $\delta$-pseudotrajectory
of $T$ with period $p$, say. For each $n \in \N_0$, let $y_n\!:= x_{n+1} - Tx_n$. Note that the sequence $(y_n)_{n \in \N_0}$ is also
periodic with period $p$. We claim that
\[
x\!:= x_0 + \sum_{j=1}^\infty S^j y^{(2)}_{j-1} - \sum_{j=0}^{p-1} \sum_{k=0}^\infty T^{kp+j} y^{(1)}_{p-j-1}
\]
is a periodic vector whose trajectory $\eps$-shadows $(x_n)_{n \in \N_0}$. Indeed,
\begin{align*}
T^px &= T^px_0 + \sum_{j=1}^p T^{p-j} y^{(2)}_{j-1}
      + \sum_{j=p+1}^\infty S^{j-p} y^{(2)}_{j-1}
      - \sum_{j=0}^{p-1} \sum_{k=1}^\infty T^{kp+j} y^{(1)}_{p-j-1}\\
     &= T^px_0 + \sum_{j=0}^{p-1} T^j y^{(2)}_{p-j-1}
      + \sum_{j=1}^\infty S^j y^{(2)}_{j-1}
      - \sum_{j=0}^{p-1} \sum_{k=0}^\infty T^{kp+j} y^{(1)}_{p-j-1}
      + \sum_{j=0}^{p-1} T^j y^{(1)}_{p-j-1}\\
     &= T^px_0 + \sum_{j=0}^{p-1} T^j y_{p-j-1}
      + \sum_{j=1}^\infty S^j y^{(2)}_{j-1}
      - \sum_{j=0}^{p-1} \sum_{k=0}^\infty T^{kp+j} y^{(1)}_{p-j-1}\\
     &= x,
\end{align*}
because
\begin{align*}
T^px_0 + \sum_{j=0}^{p-1} T^j y_{p-j-1}
  &= T^px_0 + \sum_{j=0}^{p-1} T^j(x_{p-j} - Tx_{p-j-1})\\
  &= \sum_{j=0}^p T^j x_{p-j} - \sum_{j=1}^p T^j x_{p-j}\\
  &= x_0.
\end{align*}
On the other hand,
\begin{equation}\label{PS3}
x_n - T^n x = \sum_{j=0}^{n-1} T^jy^{(1)}_{n-j-1} - \sum_{j=1}^\infty S^jy^{(2)}_{n+j-1}
   + \sum_{j=0}^{p-1} \sum_{k=0}^\infty T^{kp+j+n} y^{(1)}_{p-j-1},
\end{equation}
for all $n \in \N_0$. In fact, the case $n = 0$ follows immediately from the definition of $x$. Assume that (\ref{PS3}) holds
for a certain $n \geq 0$. Then,
\begin{align*}
x_{n+1} &- T^{n+1}x
   = y_n + T(x_n - T^nx)\\
  &= y_n + \sum_{j=0}^{n-1} T^{j+1}y^{(1)}_{n-j-1}
   - \sum_{j=1}^\infty S^{j-1}y^{(2)}_{n+j-1}
   + \sum_{j=0}^{p-1} \sum_{k=0}^\infty T^{kp+j+n+1} y^{(1)}_{p-j-1}\\
  &= y^{(1)}_n + y^{(2)}_n + \sum_{j=1}^n T^jy^{(1)}_{n-j}
   - y^{(2)}_n - \sum_{j=1}^\infty S^jy^{(2)}_{n+j}
   + \sum_{j=0}^{p-1} \sum_{k=0}^\infty T^{kp+j+n+1} y^{(1)}_{p-j-1}\\
  &= \sum_{j=0}^n T^jy^{(1)}_{(n+1)-j-1}
   - \sum_{j=1}^\infty S^jy^{(2)}_{(n+1)+j-1}
   + \sum_{j=0}^{p-1} \sum_{k=0}^\infty T^{kp+j+(n+1)} y^{(1)}_{p-j-1},
\end{align*}
proving that (\ref{PS3}) also holds with $n+1$ in place of $n$. Now, by (\ref{PS1}), (\ref{PS2}) and (\ref{PS3}),
\[
\|x_n - T^nx\| < \frac{3 \alpha \beta \delta}{1-t} = \eps \ \ \text{ for all } n \in \N_0,
\]
which completes the proof.
\end{proof}

\begin{corollary}\label{PSCor}
Every invertible generalized hyperbolic operator $T$ on a Banach space $X$ has the periodic shadowing property.
\end{corollary}

Let us recall the following basic fact (see \cite[Lemma~19]{NBerAMes21}, for instance).

\begin{lemma}\label{Unilateral1}
If $(w_n)_{n \in \N}$ is a bounded sequence of scalars, then the following assertions are equivalent:
\begin{itemize}
\item [\rm   (i)] $\displaystyle \lim_{n \to \infty} \sup_{k \in \N} |w_k w_{k+1} \cdots w_{k+n}|^\frac{1}{n} < 1$;
\item [\rm  (ii)] $\displaystyle \sup_{k \in \N} \sum_{n=0}^\infty |w_k w_{k+1} \cdots w_{k+n}| < \infty$;
\item [\rm (iii)] $\displaystyle \sup_{k \in \N} \sum_{n=0}^{k-1} |w_k w_{k-1} \cdots w_{k-n}| < \infty$.
\end{itemize}
\end{lemma}

We shall now prove that positive shadowing and positive periodic shadowing coincide for unilateral weighted backward shifts
on the classical Banach sequence spaces $\ell_p(\N)$ ($1 \leq p < \infty$) and $c_0(\N)$.

\begin{theorem}\label{Characterization1}
Let $X\!:= \ell_p(\N)$ $(1 \leq p < \infty)$ or $X\!:= c_0(\N)$.
Let $w\!:= (w_n)_{n \in \N}$ be a bounded sequence of nonzero scalars and consider the unilateral weighted backward shift
\[
B_w : (x_1,x_2,x_3,\ldots) \in X \mapsto (w_2x_2,w_3x_3,w_4x_4,\ldots) \in X.
\]
The following assertions are equivalent:
\begin{itemize}
\item [\rm   (i)] $B_w$ has the positive shadowing property;
\item [\rm  (ii)] $B_w$ has the positive periodic shadowing property;
\item [\rm (iii)] $B_w$ is generalized hyperbolic;
\item [\rm  (iv)] One of the following conditions holds:
      \begin{itemize}
      \item [\rm (a)] $\displaystyle \lim_{n \to \infty} \sup_{k \in \N} |w_k w_{k+1} \cdots w_{k+n}|^\frac{1}{n} < 1$;
      \item [\rm (b)] $\displaystyle \lim_{n \to \infty} \inf_{k \in \N} |w_k w_{k+1} \cdots w_{k+n}|^\frac{1}{n} > 1$.
      \end{itemize}
\end{itemize}
\end{theorem}

The equivalences (i) $\Leftrightarrow$ (iii) $\Leftrightarrow$ (iv) can be found in \cite{NBerAMesArxiv}.
Our goal here is to include (ii) among these equivalences.
For this purpose, we will adapt an argument used in the proof of \cite[Theorem~18]{NBerAMes21},
but taking care to construct a $\delta$-pseudotrajectory which is periodic.

\begin{proof}
By Theorem~\ref{PS}, (iii) $\Rightarrow$ (ii). Suppose that (ii) holds and let us prove that (iv) must be true.
We assume that (a) is false and prove that (b) holds. Let $\delta > 0$ be associated to $\eps\!:= 1$ in the definition of positive periodic
shadowing. By Lemma~\ref{Unilateral1}, there are integers $k \geq 2$ and $m_0 \geq 1$ such that
\begin{equation}
\sum_{n=0}^{m_0} |w_k w_{k+1} \cdots w_{k+n}| \geq \frac{1+\delta}{\delta^2}\cdot \label{C1b}
\end{equation}
Fix $m > m_0$ and let $\theta_j \in \R$ satisfy
\[
e^{i\theta_j} w_k w_{k+1} \cdots w_{k+m-j} = |w_k w_{k+1} \cdots w_{k+m-j}| \ \ \ \ (0 \leq j \leq m).
\]
Define
\begin{align*}
x_0 &\!:= \delta e^{i\theta_0} e_{k+m},\\
x_1 &\!:= B_w(x_0) + \delta e^{i\theta_1} e_{k+m-1},\\
x_2 &\!:= B_w(x_1) + \delta e^{i\theta_2} e_{k+m-2},\\
\ & \ \ \vdots \ \\
x_m &\!:= B_w(x_{m-1}) + \delta e^{i\theta_m} e_k,\\
x_{m+1} &\!:= B_w(x_m),\\
x_{m+2} &\!:= B_w(x_{m+1}),\\
\ & \ \ \vdots \ \\
x_{m+k} &\!:= B_w(x_{m+k-1}) = 0.
\end{align*}
Since
\[
(x_n)_{n \in \N_0}\!:= (x_0,\ldots,x_{m+k},x_0,\ldots,x_{m+k},x_0,\ldots,x_{m+k},\ldots)
\]
is a periodic $\delta$-pseudotrajectory of $T$, there exists $a\!:= (a_n)_{n \in \N} \in X$ such that
\begin{equation}
\|x_n - B_w^n(a)\| < 1 \ \ \text{ for all } n \in \N_0. \label{C2b}
\end{equation}
Write $a_{k+m} = \delta e^{i\theta_0} + \gamma$ with $|\gamma| < 1$. Since the $(k-1)^\text{th}$ coordinate of $x_{m+1}$ is equal to
\[
\big(|w_k w_{k+1} \cdots w_{k+m}| + |w_k w_{k+1} \cdots w_{k+m-1}| + \cdots + |w_k w_{k+1}| + |w_k|\big)\delta
\]
and the $(k-1)^\text{th}$ coordinate of $B_w^{m+1}(a)$ is $w_k w_{k+1} \cdots w_{k+m} (\delta e^{i\theta_0} + \gamma)$,
(\ref{C2b}) gives
\begin{equation}
\big|\big(|w_k w_{k+1} \cdots w_{k+m-1}| + \cdots + |w_k w_{k+1}| + |w_k|\big) \delta
 - w_k w_{k+1} \cdots w_{k+m}\gamma\big| < 1.\label{C3b}
\end{equation}
By (\ref{C1b}) and (\ref{C3b}), $|w_k w_{k+1} \cdots w_{k+m}| > 1/\delta$.
Hence, by dividing both sides of (\ref{C3b}) by $|w_k w_{k+1} \cdots w_{k+m}|\delta$, we get
\begin{equation}
\frac{1}{|w_{k+m}|} + \frac{1}{|w_{k+m}w_{k+m-1}|} + \cdots +
\frac{1}{|w_{k+m}w_{k+m-1} \cdots w_{k+1}|} < 1 + \frac{1}{\delta}\,\cdot \label{C4b}
\end{equation}
Since this holds for every $m > m_0$, we have that $\inf_{n \in \N} |w_n| > 0$. Let
\[
v_n\!:= w_n^{-1} \ (n \in \N), \ \ t\!:= k+m \ \ \text{ and } \ \ C\!:= \sum_{n=0}^{k-1} |v_k v_{k-1} \cdots v_{k-n}|.
\]
By (\ref{C4b}),
\begin{align*}
\sum_{n=0}^{t-1} |v_t v_{t-1} \cdots v_{t-n}|
  &= \sum_{n=0}^{m-1} |v_t v_{t-1} \cdots v_{t-n}| + \sum_{n=m}^{t-1} |v_t v_{t-1} \cdots v_{t-n}|\\
  &< \Big(1 + \frac{1}{\delta}\Big) + |v_t v_{t-1} \cdots v_{t-m+1}| \sum_{n=0}^{k-1} |v_k v_{k-1} \cdots v_{k-n}|\\
  &< \Big(1 + \frac{1}{\delta}\Big) \Big(1 + C\Big).
\end{align*}
Since this holds for all $t > k + m_0$, Lemma~\ref{Unilateral1} gives
\[
\lim_{n \to \infty} \sup_{t \in \N} |v_t v_{t+1} \cdots v_{t+n}|^\frac{1}{n} < 1,
\]
which is equivalent to (b).
\end{proof}

On the other hand, we will see below that the notions of shadowing and periodic shadowing do not coincide in general
for invertible operators on Banach spaces. For this purpose, we will need the result below, which gives us another class
of operators that have the periodic shadowing property. We denote by $S_X$ the unit sphere of the Banach space $X$.

\begin{theorem}\label{PeriodicThm}
Suppose that $T$ be an invertible continuous linear operator on a Banach space $X$ for which there is a direct sum decomposition
\[
X = M \oplus N,
\]
where $M$ and $N$ are closed subspaces of $X$ with $T(M) \subset M$ and $T^{-1}(N) \subset N$ such that
both $T|_M$ and $T^{-1}|_N$ are uniformly positively expansive, i.e., there are $n,m \in \N$ such that
\begin{equation}\label{Periodic1}
\|T^n y\| \geq 2 \ \text{ and } \ \|T^{-m}z\| \geq 2 \ \text{ for all } y \in S_M \text{ and } z \in S_N.
\end{equation}
For each $x \in X$ and each $k \in \Z$, let $x = x^{1,k} + x^{2,k}$ be the unique decomposition of $x$ with
$x^{1,k} \in T^k(M)$ and $x^{2,k} \in T^k(N)$. Suppose also that there is a constant $c > 0$ such that
\begin{equation}\label{Periodic2}
\|x^{1,k}\| \leq c \|x\| \ \text{ and } \ \|x^{2,k}\| \leq c \|x\| \ \text{ for all } x \in X \text{ and } k \in \Z.
\end{equation}
Then $T$ has the periodic shadowing property.
\end{theorem}

\begin{proof}
Without loss of generality, we may assume $m = n$ in (\ref{Periodic1}).
By arguing as in the proof of Proposition~\ref{PowerShad} in the Appendix,
we see that $T$ has the periodic shadowing property if and only if so does $T^n$.
Therefore, we may assume $n = 1$. Fix $\eps > 0$ and let $\delta\!:= \frac{\eps}{12c} > 0$.
Let $(x_j)_{j \in \Z}$ be a periodic $\delta$-pseudotrajectory of $T$. We claim that
\begin{equation}\label{Periodic3}
\|x_j^{2,0}\| < \frac{\eps}{2} \ \ \text{ for all } j \in \Z.
\end{equation}
Indeed, suppose that this is false and choose $\ell \in \Z$ such that
\begin{equation}\label{Periodic4}
\|x_\ell^{2,0}\| \geq \frac{\eps}{2}\,\cdot
\end{equation}
We shall prove by induction that
\begin{equation}\label{Periodic5}
\|x_{\ell-k}^{2,-k}\| \geq \frac{2^k \eps}{3} + \frac{\eps}{6} \ \ \text{ for all } k \in \N_0.
\end{equation}
The case $k = 0$ is exactly (\ref{Periodic4}). Assume that (\ref{Periodic5}) holds for a certain $k \in \N_0$. Since
\[
\|T x_{\ell-k-1}^{2,-k-1} - x_{\ell-k}^{2,-k}\| = \|(T x_{\ell-k-1})^{2,-k} - x_{\ell-k}^{2,-k}\|
  \leq c\, \|T x_{\ell-k-1} - x_{\ell-k}\| \leq c\, \delta,
\]
we obtain
\[
\frac{2^k \eps}{3} + \frac{\eps}{6} - c\,\delta \leq \|x_{\ell-k}^{2,-k}\| - c\,\delta \leq \|T x_{\ell-k-1}^{2,-k-1}\|
  \leq \frac{1}{2} \|x_{\ell-k-1}^{2,-k-1}\|,
\]
and so
\[
\|x_{\ell-k-1}^{2,-k-1}\| \geq \frac{2^{k+1} \eps}{3} + \frac{\eps}{3} - 2c\,\delta = \frac{2^{k+1} \eps}{3} + \frac{\eps}{6}\,\cdot
\]
Hence, (\ref{Periodic5}) holds with $k+1$ instead of $k$. By (\ref{Periodic2}) and (\ref{Periodic5}),
\[
\|x_{\ell-k}\| \geq \frac{1}{c}\, \|x_{\ell-k}^{2,-k}\| \geq \frac{1}{c} \Big(\frac{2^k \eps}{3} + \frac{\eps}{6}\Big) \ \
  \text{ for all } k \in \N_0.
\]
Since the sequence $(x_j)_{j \in \Z}$ is periodic, we have a contradiction. This proves that (\ref{Periodic3}) holds.
A similar argument shows that
\begin{equation}\label{Periodic6}
\|x_j^{1,0}\| < \frac{\eps}{2} \ \ \text{ for all } j \in \Z.
\end{equation}
By (\ref{Periodic3}) and (\ref{Periodic6}), $\|x_j\| < \eps$ for all $j \in \Z$.
Hence, the periodic $\delta$-pseudotrajectory $(x_j)_{j \in \Z}$ is $\eps$-shadowed by the trajectory of the zero vector,
proving that $T$ has the periodic shadowing property.
\end{proof}

Let $X\!:= \ell_p(\Z)$ $(1 \leq p < \infty)$ or $X\!:= c_0(\Z)$. Let $w\!:= (w_n)_{n \in \Z}$ be a bounded sequence of scalars
with $\inf_{n \in \Z} |w_n| > 0$ and consider the bilateral weighted backward shift
\[
B_w : (x_n)_{n \in \Z} \in X \mapsto (w_{n+1}x_{n+1})_{n \in \Z} \in X.
\]
It was proved in \cite[Theorem~18]{NBerAMes21} that $B_w$ has the shadowing property if and only if
one of the following conditions holds:
\begin{itemize}
\item [\rm (A)] $\displaystyle \lim_{n \to \infty} \sup_{k \in \Z} |w_k \cdots w_{k+n}|^\frac{1}{n} < 1$.
\item [\rm (B)] $\displaystyle \lim_{n \to \infty} \inf_{k \in \Z} |w_k \cdots w_{k+n}|^\frac{1}{n} > 1$.
\item [\rm (C)] $\displaystyle \lim_{n \to \infty} \sup_{k \in \N} |w_{-k-n} \cdots w_{-k}|^\frac{1}{n} < 1$
                and
                $\displaystyle \lim_{n \to \infty} \inf_{k \in \N} |w_k \cdots w_{k+n}|^\frac{1}{n} > 1$.
\end{itemize}
Note that (A) and (B) are exactly the cases in which $B_w$ is hyperbolic.
In case (C), $B_w$ is not hyperbolic but it is generalized hyperbolic.
It follows immediately from the above result that:
\[
B_w \textit{ has the shadowing property if and only if it is generalized hyperbolic.}
\]
In view of Corollary~\ref{PSCor}, we conclude that:
\[
\textit{If } B_w \textit{ has the shadowing property, then it has the periodic shadowing property.}
\]
We shall now see that the converse of this fact is false.

\begin{corollary}\label{BWBSNot}
Let $X\!:= \ell_p(\Z)$ ($1 \leq p < \infty$) or $X\!:= c_0(\Z)$. Let $w\!:= (w_n)_{n \in \Z}$ be a bounded sequence of scalars
with $\inf_{n \in \Z} |w_n| > 0$ and consider the bilateral weighted backward shift
\[
B_w : (x_n)_{n \in \Z} \in X \mapsto (w_{n+1}x_{n+1})_{n \in \Z} \in X.
\]
If
\begin{equation}\label{BWBSNotEq}
\lim_{n \to \infty} \inf_{k \in \N} |w_{-k-n+1} \cdots w_{-k}| = \infty
\ \ \text{ and } \ \
\lim_{n \to \infty} \sup_{k \in \N} |w_k \cdots w_{k+n-1}| = 0,
\end{equation}
then $B_w$ has the periodic shadowing property but does not have the shadowing property.
\end{corollary}

As a concrete example, we can take a weight sequence of the form
\[
w\!:= (\ldots,a,a,a,a^{-1},a^{-1},a^{-1},\ldots), \ \ \text{ where } a > 1.
\]

\begin{proof}
By the above-mentioned result from \cite{NBerAMes21}, $B_w$ does not have the shadowing property.
On the other hand, let
\[
M\!:= \{(x_n)_{n \in \Z} \in X : x_n = 0 \text{ for all } n \geq 0\}
\]
and
\[
N\!:= \{(x_n)_{n \in \Z} \in X : x_n = 0 \text{ for all } n < 0\},
\]
which are closed subspaces of $X$ such that $X = M \oplus N$, $B_w(M) \subset M$ and $B_w^{-1}(N) \subset N$.
The conditions in (\ref{BWBSNotEq}) imply that both $B_w|_M$  and $B_w^{-1}|_N$ are uniformly positively expansive.
Since condition (\ref{Periodic2}) holds with $c\!:= 1$, Theorem~\ref{PeriodicThm} guarantees that
$B_w$ has the periodic shadowing property.
\end{proof}

%%%%%%%%%%%%%%%%%%%%%%%%%%%%%%%%%%%%%%%%%%%%%%%%%%%%%%%%%%%%%%%

\section*{Appendix: Generalities on shadowing and chain recurrence for operators}\label{Generalities}

Throughout this appendix, $X$ denotes an arbitrary topological vector space over $\K$, unless otherwise specified.
We emphasize that $X$ is not assumed to be a Hausdorff space.
Recall that a set $A \subset X$ is {\em balanced} if $\lambda A \subset A$ whenever $|\lambda| \leq 1$.
We denote by $\cV_X$ the set of all balanced neighborhoods of $0$ in $X$.
It is well known that every neighborhood of $0$ in $X$ contains an element of $\cV_X$.
We denote by $L(X)$ the set of all continuous linear operators on $X$ and by
$GL(X)$ be the set of those operators that have a continuous inverse.

Since $X$ has a canonical underlying uniform structure, the notion of pseudotrajectory in the uniform space setting given in
Section~\ref{Chaos} can be rewritten as follows in the present context:
Given a neighborhood $V$ of $0$ in $X$, a {\em $V$-pseudotrajectory} of a map $f : X \to X$ is a finite or infinite sequence
$(x_j)_{i < j < k}$ in $X$ such that
\[
f(x_j) - x_{j+1} \in V \ \ \text{ for all } i < j < k-1.
\]
Finite $V$-pseudotrajectories are also called {\em $V$-chains}.
With these notions at hand, the concepts of positive shadowing, shadowing, finite shadowing, chain recurrence,
chain transitivity and chain mixing are defined in the obvious way.
We observe that if $X$ is metrizable and we endow $X$ with a compatible invariant metric,
then these notions coincide with the corresponding ones in the metric space setting.
Clearly, it is always true that
\[
\text{chain mixing} \ \ \ \Rightarrow \ \ \ \text{chain transitivity} \ \ \ \Rightarrow \ \ \ \text{chain recurrence}.
\]
The fact that these notions always coincide in the linear setting was observed in \cite{AlvBerMes21} (see also \cite{AntManVar22}):

\begin{proposition}\label{CR-CT-CM}
For any linear operator $T : X \to X$ (not necessarily continuous), the following assertions are equivalent:
\begin{itemize}
\item [\rm   (i)] $T$ is chain recurrent;
\item [\rm  (ii)] $T$ is chain transitive;
\item [\rm (iii)] $T$ is chain mixing.
\end{itemize}
\end{proposition}

Given $T : X \to X$ and $x,y \in X$, we write $x \cR y$ if for every $V \in \cV_X$, there exist $V$-chains for $T$ from $x$ to $y$
and from $y$ to $x$. With this notation, the chain recurrent set of $T$ can be written as $CR(T) = \{x \in X : x \cR x\}$. Restricted to
$CR(T)$, the relation $\cR$ is an equivalence relation and its equivalence classes are called the {\em chain recurrent classes} of $T$.

Chain recurrence is closely connected to the notion of recurrence, a property that is deserving special attention for linear dynamics
in recent years (see, e.g., \cite{CMP2014,YW2018,ABK2022,BGLP2022,CM2022,GL2023}).
A continuous linear operator $T : X \to X$ is said to be {\em recurrent} if, for every non-empty open set $U\subset X$,
there exists $k \in \N$ such that $T^k(U) \cap U \neq \emptyset$.
By a {\em recurrent vector} $x$ for $T$ we mean that, for any neighborhood $U$ of $x$, there exists $k\in\N$ with $T^kx\in U$,
and the set of recurrent vectors of $T$ is denoted by $\rec (T)$.  We easily have that any recurrent operator is chain recurrent.
Indeed, if $x \in X$ and $V \in \cV_X$, there exists $W \in \cV_X$ open with $W \subset V \cap T^{-1}(V)$.
We set $U\!:=x+W$ and find $k \in \N$ and $y \in U$ such that $T^ky \in U$.
Therefore, $(x,Ty,\dots ,T^{k-1}y,x)$ is a $V$-chain for $T$ from $x$ to itself.
Since $x$ and $V$ were arbitrary, we conclude that  $CR(T)=X$.

\begin{proposition}
For any linear operator $T : X \to X$ (not necessarily continuous), the set $CR(T)$ is a subspace of $X$ and is the unique
chain recurrent class of $T$.
\end{proposition}

\begin{proof}
Take $x,y \in CR(T)$ and $V \in \cV_X$. Choose $U \in \cV_X$ with $U+U+U+U \subset V$.
There are $U$-chains $(x_j)_{j=0}^r$ and $(y_j)_{j=0}^s$ for $T$ with $x_0 = x_r = x$ and $y_0 = y_s = y$.
Since the set $F\!:= \{x_1,\ldots,x_r,y_1,\ldots,y_s\}$ is bounded, there is $t \in \N$ such that $F \subset \lambda U$
whenever $|\lambda| \geq t$. Choose an integer $k \geq t$ which is a multiple of both $r$ and $s$. Let
\[
(x'_j)_{j=0}^k\!:= (x_0,x_1,\ldots,x_r,x_1,\ldots,x_r,\ldots,x_1,\ldots,x_r),
\]
which is also a $U$-chain for $T$ from $x$ to itself. Similarly,
\[
(y'_j)_{j=0}^k\!:= (y_0,y_1,\ldots,y_s,y_1,\ldots,y_s,\ldots,y_1,\ldots,y_s)
\]
is also a $U$-chain for $T$ from $y$ to itself. For each $0 \leq j \leq k$, let $z_j\!:= \Big(1 - \frac{j}{k}\Big) x'_j + \frac{j}{k}\, y'_j$.
Then $z_0 = x$, $z_k = y$ and
\[
Tz_j - z_{j+1}
  = \Big(1-\frac{j}{k}\Big)\big(Tx'_j - x'_{j+1}\big) + \frac{j}{k}\big(Ty'_j - y'_{j+1}\big) + \frac{1}{k}\, x'_{j+1} - \frac{1}{k}\, y'_{j+1}
  \in V,
\]
for all $0 \leq j < k$. Thus, $(z_j)_{j=0}^k$ is a $V$-chain for $T$ from $x$ to $y$. This proves that $CR(T)$ is a chain recurrent class.

Let $x,y \in CR(T)$ and $a,b \in \K$. Given $V \in \cV_X$, choose $U \in \cV_X$ with $aU + bU \subset V$.
There are $U$-chains $(x_j)_{j=0}^k$ and $(y_j)_{j=0}^t$ for $T$ from $x$ to itself and from $y$ to itself, respectively,
and we may assume $k = t$. Hence, $(ax_j + by_j)_{j=0}^k$ is a $V$-chain for $T$ from $ax+by$ to itself,
proving that $ax+by \in CR(T)$. This shows that $CR(T)$ is a subspace of $X$.
\end{proof}

It is worth to note that, in contrast with the above situation, one cannot ensure that the set $\rec(T)$ is a subspace of $X$, in general.
This is related to the problem of the recurrence of $n$-direct sum $T\oplus \dots \oplus T$ for every $n\in\N$ (see \cite{CMP2014,GLP}).

\begin{proposition}\label{InvClosedSubsp}
For any $T \in L(X)$, the set $CR(T)$ is a $T$-invariant closed subspace of~$X$. Moreover, if $T \in GL(X)$, then
$CR(T^{-1}) = CR(T)$ and $T(CR(T)) = CR(T)$; in particular, $T^{-1}$ is chain recurrent if and only if so is $T$.
\end{proposition}

\begin{proof}
$CR(T)$ is $T$-invariant: Let $x \in CR(T)$ and $V \in \cV_X$.
By the continuity of $T$, there exists $U \in \cV_X$ with $T(U) \subset V$.
Since $x \in CR(T)$, there is a $U$-chain $(x_j)_{j=0}^k$ for $T$ from $x$ to itself.
Hence, $(Tx_j)_{j=0}^k$ is a $V$-chain for $T$ from $Tx$ to itself, proving that $Tx \in CR(T)$.

\smallskip\noindent
$CR(T)$ is closed: Let $x \in \ov{CR(T)}$ and $V \in \cV_X$.
Choose $U \in \cV_X$ with $U + U \subset V$ and $W \in \cV_X$ with $W \subset U$ and $T(W) \subset U$.
Take an $y \in (x+W) \cap CR(T)$ and let $(y_j)_{j=0}^k$ be a $W$-chain for $T$ from $y$ to itself.
Then $(x,y_1,\ldots,y_{k-1},x)$ is a $V$-chain for $T$ from $x$ to itself, proving that $x \in CR(T)$.

\smallskip\noindent
$CR(T^{-1}) = CR(T)$: Let $x \in CR(T)$ and $V \in \cV_X$. Choose $U \in \cV_X$ with $T^{-1}(U) \subset V$
and let $(x_j)_{j=0}^k$ be a $U$-chain for $T$ from $x$ to itself. Since
\[
T^{-1}x_{j+1} - x_j = T^{-1}(-(Tx_j - x_{j+1})) \in T^{-1}(U) \subset V \ \text{ for all } 0 \leq j < k,
\]
we have that $(x_k,x_{k-1},\ldots,x_1,x_0)$ is a $V$-chain for $T^{-1}$ from $x$ to itself. Thus, $x \in CR(T^{-1})$.

\smallskip\noindent
$T(CR(T)) = CR(T)$: By what we have seen above,
\[
T(CR(T)) \subset CR(T) \ \text{ and } \ T^{-1}(CR(T)) = T^{-1}(CR(T^{-1})) \subset CR(T^{-1}) = CR(T),
\]
which implies the desired equality.
\end{proof}

The simplest example of a chain recurrent operator is the identity operator.
The next result gives classes of operators that are not chain recurrent.

\begin{proposition}\label{NotCR}
If either
\begin{itemize}
\item [(a)] $T \in L(X)$, $V \in \cV_X$ is convex, $V \neq X$, $\lambda \in (0,1)$ and $T(V) \subset \lambda V$, or
\item [(b)] $T \in GL(X)$, $V \in \cV_X$ is convex, $V \neq X$, $\lambda \in (1,\infty)$ and $T(V) \supset \lambda V$,
\end{itemize}
then
\[
CR(T) \subset \bigcap_{n=1}^\infty \frac{1}{n} V.
\]
In particular, $T$ is not chain recurrent. Moreover, $CR(T) = \{0\}$ if $\bigcap_{n=1}^\infty \frac{1}{n} V = \{0\}$.
\end{proposition}

\begin{proof}
Without loss of generality, we may assume that $V$ is closed. If (a) holds, choose $\delta \in (0,1)$ such that $\lambda + \delta < 1$.
Given $x \in X \backslash \bigcap_{n=1}^\infty \frac{1}{n} V$, define
\[
a\!:= \inf\{b > 0 : x \in bV\}.
\]
We have that $a > 0$, $x \in aV$ and $x \not\in bV$ for every $b \in (0,a)$.
Let $U\!:= a \delta V \in \cV_X$ and let $(x_j)_{j=0}^k$ be any $U$-chain for $T$ starting at $x_0 = x$.
If $j \in \{0,\ldots,k-1\}$ and $x_j \in aV$, then
\[
x_{j+1} = Tx_j - (Tx_j - x_{j+1}) \in a \lambda V + a \delta V = a (\lambda + \delta) V \subset a V.
\]
Hence, by induction, $x_j \in a (\lambda + \delta) V$ for all $j \in \{1,\ldots,k\}$.
In particular, $x_k \neq x$, proving that $x \not\in CR(T)$.
Case (b) follows from case (a) and the fact that $CR(T^{-1}) = CR(T)$.
\end{proof}

\begin{corollary}\label{CRcontractionexpansion}
Suppose that $X$ is a locally convex space whose topology is defined by a family $(q_i)_{i \in I}$ of semi-norms, where no $q_i$ is identically zero. If either
\begin{itemize}
\item [(a)] $T \in L(X)$, $i \in I$, $\lambda \in (0,1)$ and $q_i(Tx) \leq \lambda\, q_i(x)$ for all $x \in X$, or
\item [(b)] $T \in GL(X)$, $i \in I$, $\lambda \in (1,\infty)$ and $q_i(Tx) \geq \lambda\, q_i(x)$ for all $x \in X$,
\end{itemize}
then
\[
CR(T) \subset \{x \in X : q_i(x) = 0\}.
\]
In particular, $T$ is not chain recurrent. Moreover, $CR(T) = \{0\}$ if $q_i$ is a norm.
\end{corollary}

\begin{proposition}\label{RotationCR}
If $T \in L(X)$, $\lambda \in \K$ and $|\lambda| = 1$, then:
\begin{itemize}
\item [(a)] $CR(\lambda T) = CR(T)$.
\item [(b)] $\lambda T$ is chain recurrent if and only if so is $T$.
\item [(c)] $\lambda T$ has the positive shadowing property if and only if so does $T$.
\end{itemize}
\end{proposition}

\begin{proof}
(a): Let $x \in CR(T)$ and $V \in \cV_X$. We may assume that $V$ is open in $X$.
Let $(x_j)_{j=0}^k$ be a $V$-chain for $T$ from $x$ to itself. Given any integer $n \geq 1$, the sequence
\begin{equation}\label{chain}
(x_0,\lambda x_1,\ldots,\lambda^k x_k,\lambda^{k+1} x_1,\ldots,\lambda^{2k} x_k,\ldots,\lambda^{(n-1)k+1} x_1,\ldots,\lambda^{nk} x_k)
\end{equation}
is a $V$-chain for $\lambda T$ from $x$ to $\lambda^{nk} x$.
If $\lambda$ corresponds to a rational rotation on the unit circle, then we choose $n$ such that $\lambda^n = 1$, and so
(\ref{chain}) is a $V$-chain for $\lambda T$ from $x$ to itself.
In the case of an irrational rotation, we choose $n$ such that $\lambda^{nk}$ is so close to $1$ that we can replace the last term
in (\ref{chain}) by $x$ and so obtain a $V$-chain for $\lambda T$ from $x$ to itself. This proves that $CR(T) \subset CR(\lambda T)$.
Hence, $CR(\lambda T) \subset CR(\lambda^{-1} \lambda T) = CR(T)$.

\smallskip\noindent
(b): It follows immediately from (a).

\smallskip\noindent
(c): Suppose that $T$ has the positive shadowing property. Given $V \in \cV_X$, let $U \in \cV_X$ be associated to $V$
according to positive shadowing. If $(x_j)_{j \in \N_0}$ is a $U$-pseudotrajectory of $\lambda T$, then
$(\lambda^{-j} x_j)_{j \in \N_0}$ is a $U$-pseudotrajectory of $T$, and so it is $V$-shadowed by the trajectory of a certain
$x \in X$ under $T$. It follows that $(x_j)_{j \in \N_0}$ is $V$-shadowed by the trajectory of $x$ under $\lambda T$,
proving that $\lambda T$ has the positive shadowing property.
\end{proof}

\begin{proposition}\label{PowerCR}
For any $T \in L(X)$, $CR(T^n) = CR(T)$ for all $n \in \N$.
\end{proposition}

\begin{proof}
Fix $n \geq 2$, $x \in X$ and $V \in \cV_X$.
If $(x_j)_{j=0}^k$ is a $V$-chain for $T^n$ from $x$ to itself, then
\[
(x_0,Tx_0,\ldots,T^{n-1}x_0,x_1,Tx_1,\ldots,T^{n-1}x_1,\ldots,x_{k-1},Tx_{k-1},\ldots,T^{n-1}x_{k-1},x_k)
\]
is a $V$-chain for $T$ from $x$ to itself. Conversely, if $U \in \cV_X$ satisfies
\[
U + T(U) + T^2(U) + \cdots + T^{n-1}(U) \subset V,
\]
$(x_j)_{j=0}^k$ is a $U$-chain for $T$ from $x$ to itself and
\[
(y_j)_{j=0}^{kn}\!:= (x_0,x_1,\ldots,x_k,x_1,\ldots,x_k,\ldots,x_1,\ldots,x_k),
\]
then $(y_0,y_n,y_{2n},\ldots,y_{kn})$ is a $V$-chain for $T^n$ from $x$ to itself.
\end{proof}

\begin{corollary}\label{PowerCRCor}
For any $T \in L(X)$, the following assertions are equivalent:
\begin{itemize}
\item [(i)] $T$ is chain recurrent;
\item [(ii)] $T^n$ is chain recurrent for some $n \in \N$.
\item [(iii)] $T^n$ is chain recurrent for every $n \in \N$.
\end{itemize}
\end{corollary}

\begin{proposition}\label{PowerShad}
For any $T \in L(X)$, the following assertions are equivalent:
\begin{itemize}
\item [(i)] $T$ has the positive shadowing property;
\item [(ii)] $T^n$ has the positive shadowing property for some $n \in \N$.
\item [(iii)] $T^n$ has the positive shadowing property for every $n \in \N$.
\end{itemize}
\end{proposition}

\begin{proof}
(i) $\Rightarrow$ (iii): It is enough to note that if $(x_j)_{j \in \N_0}$ is a $U$-pseudotrajectory of $T^n$, then
\[
(x_0,Tx_0,\ldots,T^{n-1}x_0,x_1,Tx_1,\ldots,T^{n-1}x_1,x_2,Tx_2,\ldots,T^{n-1}x_2,\ldots)
\]
is a $U$-pseudotrajectory of $T$.

\smallskip\noindent
(iii) $\Rightarrow$ (ii): Obvious.

\smallskip\noindent
(ii) $\Rightarrow$ (i): Given $V \in \cV_X$, let $V' \in \cV_X$ be such that
\[
V' + T(V') + T^2(V') + \cdots + T^{n-1}(V') \subset V.
\]
Let $U' \in \cV_X$ be associated to $V'$ according to the hypothesis that $T^n$ has the positive shadowing property.
We may assume that $U' \subset V'$. Let $U \in \cV_X$ be such that
\[
U + T(U) + T^2(U) + \cdots + T^{n-1}(U) \subset U'.
\]
If $(x_j)_{j \in \N_0}$ is a $U$-pseudotrajectory of $T$, then $(x_{jn})_{j \in \N_0}$ is a $U'$-pseudotrajectory of $T^n$,
and so it is $V'$-shadowed by the trajectory of a certain $x \in X$ under $T^n$.
It follows that $(x_j)_{j \in \N_0}$ is $V$-shadowed by the trajectory of $x$ under $T$.
\end{proof}

\begin{proposition}\label{InverseShad}
If $T \in GL(X)$, then $T^{-1}$ has the (finite) shadowing property if and only if so does~$T$.
\end{proposition}

\begin{proof}
Suppose that $T$ has the shadowing property. Given $V \in \cV_X$, let $U \in \cV_X$ be associated to $V$ according to the
definition of shadowing. Choose $W \in \cV_X$ with $T(W) \subset U$. If $(x_j)_{j \in \Z}$ is a $W$-pseudotrajectory
of $T^{-1}$, then $(x_{-j+1})_{j \in \Z}$ is a $U$-pseudotrajectory of $T$, and so it is $V$-shadowed by the trajectory of some
$x \in X$ under~$T$. Hence, $(x_j)_{j \in \Z}$ is $V$-shadowed by the trajectory of $Tx$ under $T^{-1}$,
proving that $T^{-1}$ has the shadowing property. The case of finite shadowing is analogous.
\end{proof}

Recall that $X$ is said to be the {\em topological direct sum} of the subspaces $M_1,\ldots,M_n$
if $X$ is the algebraic direct sum of $M_1,\ldots,M_n$ and the canonical algebraic isomorphism
\[
(y_1,\ldots,y_n) \mapsto y_1 + \cdots + y_n
\]
is a homeomorphism from the product space $M_1 \times \cdots \times M_n$ onto $X$.

\begin{proposition}\label{DSCR}
Let $T \in L(X)$. If
\[
X = M_1 \oplus \cdots \oplus M_n
\]
is a topological direct sum of $T$-invariant subspaces $M_1,\ldots,M_n$, then:
\begin{itemize}
\item [(a)] $CR(T) = CR(T|_{M_1}) \oplus \cdots \oplus CR(T|_{M_n})$.
\item [(b)] $CR(T|_{M_i}) = CR(T) \cap M_i$ for all $i \in \{1,\ldots,n\}$.
\item [(c)] $T$ is chain recurrent if and only if so are $T|_{M_1},\ldots,T|_{M_n}$.
\item [(d)] $T$ has the positive shadowing property if and only if so do $T|_{M_1},\ldots,T|_{M_n}$.
\end{itemize}
\end{proposition}

\begin{proof}
(a): Since $CR(T|_{M_i}) \subset CR(T)$ for all $i \in \{1,\ldots,n\}$, we have that
\[
CR(T|_{M_1}) + \cdots + CR(T|_{M_n}) \subset CR(T).
\]
Conversely, let $x \in CR(T)$ and write $x = y_1 + \cdots + y_n$, where $y_1 \in M_1,\ldots,y_n \in M_n$.
We fix $i \in \{1,\ldots,n\}$ and prove that $y_i \in CR(T|_{M_i})$.
For this purpose, let $P_i : X \to M_i$ be the canonical projection.
Given $U \in \cV_{M_i}$, we have that $V\!:= P_i^{-1}(U) \in \cV_X$, because $P_i$ is continuous.
Let $(x_j)_{j=0}^k$ be a $V$-chain for $T$ from $x$ to itself. Since
\[
(T|_{M_i})(P_i x_j) - P_i x_{j+1} = P_i(Tx_j - x_{j+1}) \in P_i(V) \subset U \ \ \ (0 \leq j < k),
\]
we have that $(P_i x_j)_{j=0}^k$ is a $U$-chain for $T|_{M_i}$ from $y_i$ to itself. Thus, $y_i \in CR(T|_{M_i})$.

\smallskip\noindent
(b) and (c): They follow immediately from (a).

\smallskip\noindent
(d): Suppose that $T$ has the positive shadowing property. Fix $i \in \{1,\ldots,n\}$ and $U \in \cV_{M_i}$.
Define $V\!:= P_i^{-1}(U) \in \cV_X$ and let $V' \in \cV_X$ be associated to $V$ according to the definition of positive shadowing.
Let $U'\!:= V' \cap M_i \in \cV_{M_i}$. If a sequence $(y_j)_{j \in \N_0}$ is a $U'$-pseudotrajectory of $T|_{M_i}$,
then it is also a $V'$-pseudotrajectory of $T$, and so it is $V$-shadowed by the trajectory of a certain $x \in X$ under $T$. Since
\[
y_j - (T|_{M_i})^j(P_i x) = P_i(y_j - T^j x) \in P_i(V) \subset U \ \ \ (j \in \N_0),
\]
we have that $(y_j)_{j \in \N_0}$ is $U$-shadowed by the trajectory of $P_i x$ under $T|_{M_i}$,
proving that $T|_{M_i}$ has the positive shadowing property.

Conversely, suppose that each $T|_{M_i}$ has the positive shadowing property.
Given $V \in \cV_X$, choose $U_1 \in \cV_{M_1},\ldots,U_n \in \cV_{M_n}$ with $U_1 + \cdots + U_n \subset V$.
Let $U'_i \in \cV_{M_i}$ be associated to $U_i$ according to the definition of positive shadowing.
Let $V'\!:= P_1^{-1}(U'_1) \cap \ldots \cap P_n^{-1}(U'_n) \in \cV_X$. Then every $V'$-pseudotrajectory of $T$ is $V$-shadowed
by a real trajectory of $T$, proving that $T$ has the positive shadowing property.
\end{proof}

Recall that a {\em topological supplement} of a subspace $M$ of $X$ is a subspace $N$ of $X$ such that $X$ is the topological direct sum of $M$ and $N$.

\begin{corollary}\label{CorCR}
Let $T \in L(X)$. If $CR(T)$ admits a $T$-invariant topological supplement, then $T|_{CR(T)}$ is chain recurrent.
\end{corollary}

\begin{proof}
Let $M\!:= CR(T)$. By hypothesis, there is a $T$-invariant subspace $N$ of $X$ such that $X = M \oplus N$ as a topological direct sum.
Since $M$ is also $T$-invariant, Proposition~\ref{DSCR}(b) gives
$CR(T|_{CR(T)}) = CR(T|_M) = CR(T) \cap M = CR(T)$, and so $T|_{CR(T)}$ is chain recurrent.
\end{proof}

\begin{proposition}\label{TopSup}
Let $T \in L(X)$ and let $M$ be a $T$-invariant subspace of $X$.
Suppose that $M$ admits a $T$-invariant topological supplement.
\begin{itemize}
\item [(a)] If $T$ is chain recurrent, then so is $T|_M$.
\item [(b)] If $T$ has the positive shadowing property, then so does $T|_M$.
\end{itemize}
\end{proposition}

\begin{proof}
It follows immediately from Proposition \ref{DSCR}(c,d).
\end{proof}

\begin{remark}
In Proposition~\ref{TopSup}, it is not enough to assume that $M$ has a topological supplement, the hypothesis of $T$-invariance is essential.
In order to give a counterexample, assume that $X$ is a separable Banach space and that $T \in GL(X)$ is generalized hyperbolic
but not hyperbolic ({\em shifted hyperbolic} in the terminology of \cite{CirGolPuj21}).
Let $X = M \oplus N$ be the direct sum decomposition given by the definition of generalized hyperbolicity.
By the spectral radius formula, there exist constants $t \in (0,1)$ and $c \geq 1$ such that
\[
\|T^ny\| \leq c\, t^n \|y\| \text{ and } \|T^{-n}z\| \leq c\, t^n \|z\| \text{ for all } n \in \N_0, y \in M, z \in N.
\]
Choose $\lambda \in (t,1)$ and a nonzero $w \in M \cap T(N)$, and define
\[
u\!:= \sum_{n=-\infty}^\infty T^n w \ \ \ \text{ and } \ \ \ v\!:= \sum_{n=-\infty}^\infty \lambda^n T^n w .
\]
We have that $u$ is a nontrivial fixed point of $T$ and $v$ is an eigenvector of $T$ associated to the eigenvalue $\lambda^{-1}$.
Therefore,
\[
F\!:= \spa\{u\}, \ \ \ G\!:= \spa\{v\} \ \ \text{ and } \ \ H\!:= \spa\{u,v\}
\]
are $T$-invariant subspaces of $X$, which admit topological supplements since they are finite-dimensional.
Moreover:
\begin{itemize}
\item $T|_F$ is chain recurrent but does not have the positive shadowing property,
\item $T|_G$ has the shadowing property but is not chain recurrent,
\item $T|_H$ neither is chain recurrent nor has the positive shadowing property.
\end{itemize}
Although $T$ has the shadowing property, it may fail to be chain recurrent.
However, by \cite[Corollary~2]{CirGolPuj21}, the restriction of $T$ to the smallest closed $T$-invariant subspace $Y$ of $X$
containing $M \cap T(N)$ satisfies the so-called {\em frequent hypercyclicity criterion} \cite[Section~9.2]{KGroAPer11}, and so
it exhibits several types of chaotic behaviors, including {\em frequent hypercyclicity}, {\em mixing}, {\em Devaney chaos},
{\em dense distributional chaos} and {\em dense mean Li-Yorke chaos} \cite{BerBonMulPer13,BerBonPer20,KGroAPer11}.
In particular, $T|_Y$ is chain recurrent. Concrete examples of shifted hyperbolic operators on the Banach spaces $c_0(\Z)$ and
$\ell_p(\Z)$ ($1 \leq p < \infty$) were obtained in \cite[Theorem~9]{NBerAMes21}.
\end{remark}

For invertible operators on Banach spaces, the next proposition shows that the closed $T$-invariant subspace $CR(T)$ has the
property that the restricted operator $T|_{CR(T)}$ has the shadowing property whenever so does $T$.

\begin{proposition}
Let $X$ be a Banach space. 
If $T \in GL(X)$ has the shadowing property, then $T|_{CR(T)}$ is chain recurrent and has the shadowing property.
\end{proposition}

\begin{proof}
We claim that
\begin{equation}\label{Fact1}
\{x \in X : (T^jx)_{j \in \N} \text{ is bounded}\} \subset I_0(T).
\end{equation}
Indeed, assume $C\!:= \sup_{j \in \N} \|T^jx\| < \infty$. 
Given $\delta > 0$, there exists a decreasing sequence $1 = t_0 > t_1 > \cdots > t_{k-1} > t_k = 0$ of real numbers
such that $t_j - t_{j+1} < \delta/C$ for all $j \in \{0,\ldots,k-1\}$.
Then, $(t_jT^jx)_{j=0}^k$ is a $\delta$-chain for $T$ from $x$ to $0$.
Similarly,
\begin{equation}\label{Fact2}
\{x \in X : (T^{-j}x)_{j \in \N} \text{ is bounded}\} \subset O_0(T).
\end{equation}
By combining (\ref{Fact1}) and (\ref{Fact2}), we obtain
\begin{equation}\label{Fact3}
\{x \in X : (T^jx)_{j \in \Z} \text{ is bounded}\} \subset CR(T).
\end{equation}

Let us prove that $T|_{CR(T)}$ has the shadowing property. 
By Theorem~\ref{EquivSh}, it is enough to show that $T|_{CR(T)}$ has the finite shadowing property.
Fix $\eps > 0$ and let $\delta > 0$ be given by the shadowing property of $T$.
Let $(x_j)_{j=0}^k$ be a $\delta$-chain for $T|_{CR(T)}$. 
Since $x_0, x_k \in CR(T)$, there are $\delta$-chains for $T$ of the forms $(0,x_{-i},\ldots,x_0)$ and $(x_k,\ldots,x_\ell,0)$.
Hence,
\[
(x_j)_{j \in \Z}\!:= (\ldots,0,0,x_{-i},\ldots,x_0,\ldots,x_k,\ldots,x_\ell,0,0,\ldots)
\]
is a $\delta$-pseudotrajectory of $T$. Thus, there exists $x \in X$ with $\|T^jx - x_j\| < \eps$ for all $j \in \Z$.
By (\ref{Fact3}), $x \in CR(T)$ and we are done.

Let us now prove that $T|_{CR(T)}$ is chain recurrent. Fix $x \in CR(T)$ and $\delta > 0$.
Let $\eta > 0$ be associated to $\delta/(1+\|T\|)$ according to the shadowing property of $T$.
Since $x \in CR(T)$, there exists an $\eta$-chain $(x_j)_{j=0}^k$ for $T$ from $x$ to itself.
As in the previous paragraph, we can extend this $\eta$-chain to a bounded $\eta$-pseudotrajectory $(x_j)_{j \in \Z}$ of $T$.
By our choice of $\eta$, there exists $y \in X$ such that
\[
\|T^jy - x_j\| < \frac{\delta}{1 + \|T\|} \ \ \text{ for all } j \in \Z.
\]
By (\ref{Fact3}), $y \in CR(T)$. Hence, $(x,Ty,T^2y,\ldots,T^{k-1}y,x)$ is a $\delta$-chain for $T|_{CR(T)}$ from $x$ to itself,
proving that $T|_{CR(T)}$ is chain recurrent.
\end{proof}

In general, properties associated with chaotic behavior are not possible for compact operators. Certainly, although chain recurrence
is equivalent to chain transitivity in linear dynamics, and the second property can be thought as a chaotic behavior, the fact that, e.g.,
the identity of a finite dimensional space is an operator which is chain recurrent and compact spoils this impossibility for chain recurrence.
The natural question is whether we can have anything else than the ``finite dimensional'' case for chain recurrent compact operators.
We will show that, as a consequence of the previous results, there is nothing else.
We recall that a linear operator $T : X \to X$ on a topological vector space $X$ is {\em compact} if there is $V \in \cV_X$ such that
$T(V)$ is relatively compact.

\begin{proposition}\label{compactop}
Let $T \in L(X)$ be a compact operator on a locally convex space $X$ with continuous norm. Then $CR(T)$ is a finite dimensional subspace.
\end{proposition}

\begin{proof}
We first suppose that the scalar field is $\C$ and $X$ is a Banach space. In this case, the spectrum of $T$ consists, at most, of $0$
and a sequence of eigenvalues of $T$ whose limit is $0$, and each eigenvalue has a finite dimensional associated eigenspace.
The Riesz decomposition theorem yields that we can write $X = M_1 \oplus M_2$ with $M_i$ closed and $T$-invariant subspace, $i=1,2$,
$M_1$ finite dimensional and consisting of sums of eigenvectors, and $M_2$ so that, for $T_2\!:=T|_{M_2}$, we have that the spectral
radius of $T_2$ is strictly less than $1$. In particular, there is $n \in \N$ such that $\norm{T^n_2} < 1$.
By Propositions \ref{NotCR} and \ref{PowerCR}, we have that
\[
CR(T_2) = CR(T_2^n) = \{0\}.
\]
Thus, by applying Proposition~\ref{DSCR}, we get $CR(T) = CR(T_1) \oplus CR(T_2) \subset M_1$, and $CR(T)$ is finite dimensional.

Still in the complex case, but now allowing $X$ to be an arbitrary locally convex space with continuous norm, we take $V \in \cV_X$ absolutely convex such that
$K\!:= \overline{T(V)}$ is compact and the gauge $p$ of $V$ is a norm. In particular, $T$ naturally induces an operator $T_V$ on the local Banach space $X_V$
(which is the completion of the normed space $(X,p)$), which is also compact.
We easily have that $CR(T) \subset CR(T_V)$, with a natural inclusion, and $CR(T)$ is finite dimensional.

Finally, when $\K = \R$, we consider the complexification $\tilde{X} = X + i X = X \oplus X$ with the (compact) operator
$\tilde{T} = T + i\,T = T\oplus T$. Proposition~\ref{DSCR} yields $CR(\tilde{T}) = CR(T) + i\,CR(T)$, and $CR(T)$ is finite dimensional.
\end{proof}

\begin{remark}
When $X$ is a complex Banach space, there is another typical property satisfied by the spectrum of $T$ when it has certain chaotic behaviour:
Namely, every connected component of $\sigma(T)$ intersects $\mathbb{T}$. This is also the case for chain recurrence.
Indeed, if $K$ is a connected component of $\sigma(T)$, again Riesz decomposition theorem brings the existence of a
$T$-invariant closed subspace $M$, with a complement which is also $T$-invariant, such that $\sigma(T|_M) = K$.
Then, by Proposition~\ref{TopSup}, $T|_M$ is also chain recurrent. If $K$ does not intersect $\sigma(T)$, either it is contained in $\D$,
which is impossible since its spectral radius would be strictly less than $1$, or it is contained in the complementary of $\overline{\D}$,
also impossible since this would mean that $T|_M$ is invertible with a spectral radius of its inverse strictly less than $1$.
\end{remark}

\begin{proposition}\label{Product}
Suppose that $X$ is the product of a family $(X_i)_{i \in I}$ of topological vector spaces over $\K$, $T_i \in L(X_i)$ for each $i \in I$,
and $T \in L(X)$ is the product operator given by
\[
T((x_i)_{i \in I})\!:= (T_i x_i)_{i \in I}.
\]
The following properties hold:
\begin{itemize}
\item [(a)] $CR(T) = \prod_{i \in I} CR(T_i)$.
\item [(b)] $T$ is chain recurrent if and only if so is each $T_i$.
\item [(c)] $T$ has the positive shadowing property if and only if so does each $T_i$.
\end{itemize}
\end{proposition}

\begin{proof}
For each $i \in I$, let $\pi_i : X \to X_i$ denote the canonical projection.

\smallskip\noindent
(a): Let $x\!:= (x_i)_{i \in I} \in \prod_{i \in I} CR(T_i)$ and $V \in \cV_X$. We may assume that $V = \prod_{i \in I} V_i$,
where $V_i \in \cV_{X_i}$ for each $i \in I$ and $V_i = X_i$ except for $i$ in a finite subset $J$ of $I$.
For each $i \in J$, there is a $V_i$-chain $(x^{(j)}_i)_{j=0}^{k_i}$ for $T_i$ from $x_i$ to itself,
and we may assume all the $k_i$'s equal to the same $k$.
For each $j \in \{0,\ldots,k\}$, let $x^{(j)}_i\!:= x_i$ for all $i \in I \backslash J$, and let $x^{(j)}\!:= (x^{(j)}_i)_{i \in I} \in X$.
Then $(x^{(j)})_{j=0}^k$ is a $V$-chain for $T$ from $x$ to itself.
For the converse, simply note that if $y \in CR(T)$, $i \in I$, $U_i \in \cV_{X_i}$ and $(y^{(j)})_{j=0}^\ell$ is a
$\pi_i^{-1}(U_i)$-chain for $T$ from $y$ to itself, then $(\pi_i(y^{(j)}))_{j=0}^\ell$ is a $U_i$-chain for $T_i$ from $\pi_i(y)$ to itself.

\smallskip\noindent
(b): It follows immediately from (a).

\smallskip\noindent
(c): Suppose that each $T_i$ has the positive shadowing property. Let $V\!:= \prod_{i \in I} V_i$, where $V_i \in \cV_{X_i}$
for each $i \in I$ and $V_i = X_i$ except for $i$ in a finite subset $J$ of $I$.
For each $i \in J$, let $U_i \in \cV_{X_i}$ be associated to $V_i$ according to the definition of positive shadowing.
Let $U\!:= \bigcap_{i \in J} \pi_i^{-1}(U_i) \in \cV_X$. Then every $U$-pseudotrajectory of $T$ is $V$-shadowed by some real trajectory
of $T$, proving that $T$ has the positive shadowing property.
Conversely, if $T$ has the positive shadowing property and $i \in I$, we regard the product space $X$ as the topological direct sum
of the $T$-invariant ``subspaces'' $X_i$ and $\prod_{\ell \neq i} X_\ell$ in a canonical way and apply Proposition~\ref{DSCR}(d)
to conclude that $T_i$ has the positive shadowing property.
\end{proof}

\begin{remark}\label{Product2}
Consider the notations of the previous proposition.
Let $Y\!:= \bigoplus_{i \in I} X_i$ be the external direct sum of the family $(X_i)_{i \in I}$, that is, the set of all $(x_i)_{i \in I} \in X$
such that $x_i = 0$ except for a finite number of indices. If we consider $Y$ as a subspace of the product topological vector space $X$
and $S \in L(Y)$ is the operator obtained by restricting $T$ to $Y$, then the following properties hold:
\begin{itemize}
\item [(a)] $CR(S) = \bigoplus_{i \in I} CR(T_i)$.
\item [(b)] $S$ is chain recurrent if and only if so is each $T_i$.
\item [(c)] $S$ has the positive shadowing property if and only if so does each $T_i$.
\end{itemize}
The proof is similar to the previous one and so we leave it to the reader.
\end{remark}

\begin{proposition}\label{DSLCS}
Suppose that $X$ is the locally convex direct sum of a family $(X_i)_{i \in I}$ of locally convex spaces over $\K$,
$T_i \in L(X_i)$ for each $i \in I$, and $T \in L(X)$ is given by
\[
T((x_i)_{i \in I})\!:= (T_i x_i)_{i \in I}.
\]
The following properties hold:
\begin{itemize}
\item [(a)] $CR(T) = \bigoplus_{i \in I} CR(T_i)$.
\item [(b)] $T$ is chain recurrent if and only if so is each $T_i$.
\item [(c)] If $T$ has the positive shadowing property, then so does each $T_i$.
\end{itemize}
\end{proposition}

\begin{proof}
(a): Let $x\!:= (x_i)_{i \in I} \in \bigoplus_{i \in I} CR(T_i)$ and $V \in \cV_X$.
There is a finite subset $J$ of $I$ such that $x_i = 0$ for all $i \in I \backslash J$.
We may regard $X$ as the topological direct sum of the $T$-invariant ``subspaces'' $Y\!:= \bigoplus_{i \in J} X_i$ and
$Z\!:= \bigoplus_{i \in I \backslash J} X_i$ in a canonical way. Let $U \in \cV_Y$ and $W \in \cV_Z$ be such that $U + W \subset V$.
Since $J$ is finite, $Y$ coincides with the product space $\prod_{i \in J} X_i$.
Hence, by Proposition~\ref{Product}(a), $(x_i)_{i \in J} \in \prod_{i \in J} CR(T_i) = CR(T|_Y)$.
Thus, there is a $U$-chain for $T|_Y$ from $(x_i)_{i \in J}$ to itself.
Each element of this $U$-chain can be regarded as an element of $X$ by completing the remaining coordinates with $0$'s.
In this way we obtain a $V$-chain for $T$ from $x$ to itself, proving that $x \in CR(T)$.
For the converse, note that each $X_i$ can be regarded as a $T$-invariant subspace of $X$ that admits a $T$-invariant
topological supplement, and so we can apply Proposition~\ref{DSCR}(a).

\smallskip\noindent
(b): It follows immediately from (a).

\smallskip\noindent
(c): It is enough to apply Proposition~\ref{TopSup}(b), since each $X_i$ can be regarded as a $T$-invariant subspace of $X$
that admits a $T$-invariant topological supplement.
\end{proof}

\begin{remark}
The converse of Proposition~\ref{DSLCS}(c) is false in general. For instance, consider the locally convex direct sum $\K^{(\N)}$,
where $\K$ is endowed with its usual topology, and
\[
T((x_n)_{n \in \N})\!:= (2x_n)_{n \in \N} \ \ \text{ for all } (x_n)_{n \in \N} \in \K^{(\N)}.
\]
We know that the operator $x \in \K \mapsto 2x \in \K$ has the shadowing property, but we will show that the operator $T$
does not have the positive shadowing property. For this purpose,  let $j_n : \K \to \K^{(\N)}$ denote the $n^\text{th}$ canonical
injection and let $\ov{\Delta}(0;\delta)\!:= \{\lambda \in \K : |\lambda| \leq \delta\}$ for $\delta > 0$.
Consider the following neighborhood of $0$ in $\K^{(\N)}$:
\[
V\!:= \co\Big(\bigcup_{n=1}^\infty j_n(\ov{\Delta}(0;1))\Big),
\]
where $\co(A)$ denotes the convex hull of the set $A \subset \K^{(\N)}$.
Given any neighborhood $U$ of $0$ in $\K^{(\N)}$ of the form
\[
U\!:= \co\Big(\bigcup_{n=1}^\infty j_n(\ov{\Delta}(0;\delta_n))\Big),
\]
with $\delta_n > 0$ for all $n \in \N$, define $x^{(0)}\!:= 0$ and $x^{(j)}\!:= T x^{(j-1)} + \delta_j e_j$ for $j \geq 1$,
where $e_j$ is the sequence whose $j^\text{th}$ coordinate is $1$ and the others are $0$.
Then $(x^{(j)})_{j \in \N_0}$ is a $U$-pseudotrajectory of $T$, but it cannot be $V$-shadowed by a trajectory of $T$,
because each $x \in \K^{(\N)}$ has finite support.
\end{remark}

\begin{remark}
Propositions \ref{RotationCR}(c), \ref{PowerShad}, \ref{DSCR}(d), \ref{TopSup}(b), \ref{Product}(c) and \ref{DSLCS}(c),
as well as Remark~\ref{Product2}(c), remain true if we replace positive shadowing by finite shadowing.
Moreover, all these results have analogous formulations with shadowing instead of positive shadowing in the case of invertible operators.
\end{remark}

We close the paper by proposing the following open problems:

\medskip\noindent
{\bf Problem A.} To characterize the Fr\'echet spaces in which shadowing and finite shadowing coincide for operators or
at least find sufficient (resp.\ necessary) conditions for the validity of this equivalence in the case of non-normable Fr\'echet spaces.

\medskip\noindent
{\bf Problem B.} Does Theorem~\ref{CRDDC} hold for every Fr\'echet space? If not, for which Fr\'echet spaces does
the property described in Theorem~\ref{CRDDC} hold?

\medskip\noindent
{\bf Problem C.} To characterize the periodic shadowing property for bilateral weighted shifts on Banach sequence spaces.

\medskip\noindent
{\bf Problem D.} If $T \in L(X)$ is an (invertible) operator on a Banach space $X$, is it true that $T|_{CR(T)}$ is always chain recurrent?

\bigskip\noindent
{\it Note:} We were informed that Antoni L\'opez-Mart\'inez and Dimitris Papathanasiou have recently solved Problem~D in the negative.

%%%%%%%%%%%%%%%%%%%%%%%%%%%%%%%%%%%%%%%%%%%%%%%%%%%%%%%%%%%%%%%

\section*{Acknowledgements}

The first author is beneficiary of a grant within the framework of the grants for the retraining, modality Mar\'ia Zambrano,
in the Spanish university system (Spanish Ministry of Universities, financed by the European Union, NextGenerationEU).
The first author was also partially supported by CNPq (Conselho Nacional de Desenvolvimento Cient\'ifico e Tecnol\'ogico - Brasil),
project {\#}308238/2021-4, and by CAPES (Coordenação de Aperfeiçoamento de Pessoal de Nível Superior - Brasil), Finance Code 001.
Both authors were partially supported by MCIN/AEI/10.13039/501100011033, Projects PID2019-105011GB-I00 and PID2022-139449NB-I00,
and the second author was also supported by Generalitat Valenciana, Project PROMETEU/2021/070. We would like to thank the referee whose careful review 
resulted in an improved presentation of the article. 

%%%%%%%%%%%%%%%%%%%%%%%%%%%%%%%%%%%%%%%%%%%%%%%%%%%%%%%%%%%%%%%

\end{document}